\documentclass[12pt]{amsart}
\usepackage{amsmath,amssymb,amsbsy,amsfonts,amsthm,latexsym,
                        amsopn,amstext,amsxtra,euscript,amscd,mathrsfs,color,bm,cite}
\usepackage{float}
\usepackage[english]{babel}
\usepackage{mathtools}
\usepackage{todonotes}
\usepackage{url}
\usepackage[colorlinks,linkcolor=blue,anchorcolor=blue,citecolor=blue,backref=page]{hyperref}
\usepackage{cases}
\usepackage{txfonts}
\usepackage{verbatim}

\def\re{\text{Re}}
\def\ma{\mathfrak{a}}
\def\r{\right}

\let\ve=\varepsilon

\let\ol=\overline

\let\vp=\varphi
\let\wt=\widetilde
\let\wh=\widehat
\theoremstyle{definition}
\newtheorem{definition}{Definition}[section]

\newtheorem*{remark}{Remark}
\newtheorem*{notation}{Notation}

\theoremstyle{plain}
\newtheorem{theorem}{Theorem}

\newtheorem{lemma}[theorem]{Lemma}
\newtheorem{conjecture}[theorem]{Conjecture}
\numberwithin{equation}{section}
\numberwithin{theorem}{section}

\def\qed{\ifhmode\textqed\fi
   \ifmmode\ifinner\quad\qedsymbol\else\dispqed\fi\fi}
\def\textqed{\unskip\nobreak\penalty50
    \hskip2em\hbox{}\nobreak\hfil\qedsymbol
    \parfillskip=0pt \finalhyphendemerits=0}
\def\dispqed{\rlap{\qquad\qedsymbol}}

\begin{document}
\title{T\MakeLowercase{he fourth moment of }D\MakeLowercase{irichlet} $L$-\MakeLowercase{functions along the critical line}}
\author{Xiaosheng Wu}
\date{}
\address {School of Mathematics, Hefei University of Technology, Hefei 230009,
P. R. China.}
\email {xswu@amss.ac.cn}
\thanks{}
\subjclass[2010]{11M06 }
\keywords{fourth moment; Dirichlet $L$-function; divisor problem; power saving; }

\begin{abstract}
For a positive integer $q\not\equiv 2 \pmod 4$, this work considers the fourth moment of Dirichlet $L$-functions averaged over both $t\in [0,T]$ and primitive characters to modulus $q$. An asymptotic formula with a power saving from both $q$-aspect and $t$-aspect in the error term is obtained.
\end{abstract}
\maketitle

\section{Introduction}

Moments of families of $L$-functions have a wide range of applications, and their computation is counted as a central problem in number theory, which may go back to Hardy and Littlewood \cite{HL16}. If we define
\begin{align}
  M_k(T)=\int_0^T|\zeta(\tfrac12+it)|^{2k}dt,\notag
\end{align}
Hardy and Littlewood proved that $M_1(T)\sim T\log T$, and then Ingham (see \cite{Tit86}; Chapter VII) showed the fourth moment to be $M_2(T)\sim \frac1{2\pi^2}T(\log T)^4$. In general, it is conjectured that
\begin{align}\label{conMk}
  M_{k}(T)\sim C_kT(\log T)^{k^2},
\end{align}
for some constant $C_k$, whose precise value was predicted by Keating and Snaith \cite{KS00} by analogies with random matrix theory. Although higher moments have not yet been computed, Soundararajan \cite{Sou09} obtained almost sharp upper bounds on GRH that $M_k(T)\ll T(\log T)^{k^2+\ve}$, and the $\ve$ on the power of $\log T$ was then removed by Harper \cite{Har13}.

Conrey, Farmer, Keating, Rubinstein and Snaith \cite{CFK+05} refined the conjecture \eqref{conMk}, and predicted that
\begin{align}
  M_k(T)=TP_{k^2}(\log T)+O(T^{\frac12+\ve}),\notag
\end{align}
where $P_{k^2}$ is a polynomial of degree $k^2$.  For $k=2$, the asymptotic formula has already been proved by Heath-Brown \cite{HB79} except for the strength of the error.
More precisely, Heath-Brown \cite{HB79} proved that
\begin{align}
  \int_0^T|\zeta(\tfrac12+it)|^4dt=TP_4(\log T)+O\left(T^{\frac78+\ve}\r).\notag
\end{align}
To deduce all main terms as well as a power saving error term is a significant challenge, and it requires a difficult analysis on off-diagonal terms to distinguish lower-order main terms. Some deep estimates on the divisor problem
\begin{align}
  \sum_{n\le x}d(n)d(n+f)\notag
\end{align}
were explored to obtain the power saving in \cite{HB79}. Further progresses on the fourth moment of the Riemann zeta-function were based on methods originating in the spectral theory of automorphic forms, in particular the Kuznetsov formula. Then, Zavorotnyi \cite{Zav89} improved the result to
\begin{align}\label{eqZav}
  \int_0^T|\zeta(\tfrac12+it)|^4dt=TP_4(\log T)+O\left(T^{\frac23+\ve}\r).
\end{align}
Motohashi established a beautiful explicit formula for a smoothed version of the fourth moment of the Riemann zeta-function in terms of the cubes of the central values of certain automorphic $L$-functions (to see Theorem 4.2 of \cite{Mot97}).  Based on this explicit formula, Ivi\'c and Motohashi \cite{IM95} were able to replace the factor $T^\ve$ in \eqref{eqZav} by a suitable power of $\log T$, and this is the best estimate to date. A generalization of Motohashi's formula to the fourth moment of Dirichlet $L$-functions weighted by a non-archimedean test function has also been obtained by Blomer, Humphries, Khan, and Milinovichet \cite{BHK+20}, which proceeds differently with some important applications.

To some extent, the fourth moment averaging over $t$ for an individual Diriclet $L$-function is a direct extension of the problem from the Riemann zeta-function.
Recently, Topacogullari \cite{Top19} considered this moment and proved that
\begin{align}\label{eqTop}
  \int_1^T|L(\tfrac12+it, \chi)|^4=\int_1^TP_{\chi}(\log t)dt+O\left(q^{2-3\theta}T^{\frac12+\theta+\ve}+qT^{\frac23+\ve}\r),
\end{align}
where $P_{\chi}$ is a polynomial of degree $4$ with coefficients depending on $q$, and where $\theta=7/{64}$ is the current best-known bound on the size of the Hecke eigenvalue of a Maass form, due to Kim and Sarnak \cite{Kim03}. With $\theta=7/{64}$, this asymptotic formula is non-trivial in the range $q\ll T^{{25}/{107}-\ve}$.

The fourth moment of Dirichlet $L$-functions at the central value
\[
\frac1{\varphi^*(q)}\mathop{\sum\nolimits^*}_{\chi ~(\bmod q)}\left|L\left(\tfrac12,\chi\r)\r|^4
\]
has gotten a lot of attention. Here, the sum is over all primitive characters modulo $q$, and $\varphi^*(q)$ is the number of these primitive characters. Due to a conjecture for the moments of unitary style in \cite{CFK+05}, it is predicted that
\begin{conjecture}\label{conjq}
For any $q\not\equiv 2 \pmod 4$, we have
\begin{align}
\frac1{\varphi^*(q)}\mathop{\sum\nolimits^*}_{\chi ~(\bmod q)}\left|L\left(\tfrac12,\chi\r)\r|^4= \prod_{p\mid q}\frac{(1-p^{-1})^3}{(1+p^{-1})}P_4(\log q)+O\left(q^{-\frac12+\ve}\r),\notag
\end{align}
where $P_4$ is a computable absolute polynomial of degree 4.
\end{conjecture}

 It was first proved by Heath-Brown \cite{HB81} that
\begin{align}\label{HBq}
\frac1{\varphi^*(q)}\mathop{\sum\nolimits^*}_{\chi ~(\bmod q)}\left|L\left(\tfrac12,\chi\r)\r|^4=\frac1{2\pi^2}\prod_{p\mid q}\frac{(1-p^{-1})^3}{(1+p^{-1})}(\log q)^4+O\left(2^{\omega(q)}\frac{q}{\varphi^*(q)}(\log q)^3\r),
\end{align}
where $\omega(q)$ means the number of distinct prime factors of $q$. This asymptotic formula is non-trivial if $\omega(q)$ is not too large. Then Soundararajan \cite{Sou07} filled in this exception with a sharper error term, so the leading term of the asymptotic formula was proved completely.

In 2011, Young \cite{You11} made an important breakthrough and proved the asymptotic formula for prime moduli that
\begin{align}\label{Youngformula}
\frac1{\varphi^*(p)}\mathop{\sum\nolimits^*}_{\chi(\bmod p)}|L(\tfrac12,\chi)|^4=P_4(\log p)+O\left(p^{-\frac1{80}+\frac{\theta}{40}+\ve}\r).
\end{align}
Then Blomer, Fouvry, Kowalski, Michel and Mili\'cevi\'c \cite{BFK+17a,BFK+17b} improved on the error term in \eqref{Youngformula} to $p^{-1/20}$, with an average result of Hecke eigenvalues to remove $\theta$, as well as some new results on bilinear forms in Kloosterman sums; see also Fouvry, Kowalski, and Michel \cite{FKM14},  Kowalski, Michel, and Sawin \cite{KMS17}, and Shparlinski and Zhang \cite{SZ16}.

By distinguishing the main terms in a special divisor sum function of type
\[
\mathcal{D}_{q}\left(s,\lambda,\frac h{l},r\r)=\sum_{\substack{(n,q)=1\\ (n+r,q)=1}}\frac{\sigma_\lambda(n)} {n^s} e\left(n\frac h{l}\r)
\]
with $\sigma_{\lambda}(n)=\sum_{d\mid n}d^\lambda$, the author \cite{Wu20} succeed in deducing the asymptotic formula for general moduli.
It is proved in \cite{Wu20} that, for any integer $q\nequiv 2 \pmod 4$,
\begin{align}\label{eqmain}
\frac1{\varphi^*(q)}\mathop{\sum\nolimits^*}_{\chi(\bmod q)}|L(\tfrac12,\chi)|^4=\prod_{p\mid q}\frac{(1-p^{-1})^3}{(1+p^{-1})}P_4(\log q)+O\left(q^{-\frac1{14}+\frac37\theta+\ve}\r).
\end{align}
In \eqref{eqmain}, there is also a considerable improvement on the error term, as a special case, it sharpens the error term to $p^{-1/14}$ for prime moduli. This is due to an application of a recent progress on bilinear forms in Kloosterman sums by Kerr, Shparlinski, Wu, and Xi \cite{KSWX22}.

Actually, the fourth moment of Dirichlet $L$-functions, including both $q$-aspect and $t$-aspect, was the first to draw attention to, which
may go back to Montgomery \cite{Mon71}, who proved that
\begin{align}
\mathop{\sum\nolimits^*}_{\chi ~(\bmod q)}\int_0^T\left|L\left(\tfrac12+it,\chi\r)\r|^4dt\ll\varphi(q)T(\log qT)^4.\notag
\end{align}
For easy of notation, we will apply
\[
T_1=T+1
\]
in place of $T$ in the error term, avoiding the case $T\rightarrow 0$.
According to the conjecture in \cite{CFK+05}, we may predict that
\begin{conjecture}\label{conjT}
For any positive integer $q\not\equiv 2 \pmod 4$ and $T>0$, there exist computable constants $c_0$, $c_1$, $c_2$, $c_3$, $c_4$ that
\begin{align}
\frac1{\varphi^*(q)}&\mathop{{\sum}^*}_{\chi~(\bmod q)}\int_0^T\left|L\left(\tfrac12+it,\chi\right)\right|^4dt=\prod_{p\mid q}\frac{(1-p^{-1})^3}{(1+p^{-1})}\notag\\
&\times\sum_{j=0}^4c_j \int_0^T\frac12\sum_{\ma=0,1} \left(\log\frac{q}{\pi} +\frac12 \frac{\Gamma'}{\Gamma}\left(\tfrac{\frac12-it+\ma}2\r) +\frac12\frac{\Gamma'}{\Gamma}\left(\tfrac{\frac12+it+\ma}2\r)\r)^j dt+O\left(T_1^{\frac12+\ve}q^{-\frac12+\ve}\r).\notag
\end{align}
\end{conjecture}

Conjecture \ref{conjT} looks a little different from Conjecture \ref{conjq} since a polynomial seems gone. By Stirling's approximation
\[
\frac{\Gamma'}{\Gamma}\left(\tfrac{\frac12\pm it+\ma}2\r)=\log\frac{t}{2}+O\left(\frac1t\r),
\]
it is easy to see that the main terms of the conjecture would evolve to
\[
T\prod_{p\mid q}\frac{(1-p^{-1})^3}{(1+p^{-1})}P_4\left(\log Tq\r)
\]
for large $T$. However, this should not be applied for small $T$ since the error in using Stirling's approximation would be large.

The leading term of the asymptotic formula has already been obtained by Rane \cite{Ran81} for some of $q$ and large $T$, he proved that
\begin{align}\label{eqRan}
\frac1{\varphi^*(q)}\mathop{\sum\nolimits^*}_{\chi ~(\bmod q)}&\int_T^{2T}\left|L\left(\tfrac12+it,\chi\r)\r|^4dt\\
&=\frac{T}{2\pi^2}\prod_{p\mid q}\frac{(1-p^{-1})^3}{(1+p^{-1})}(\log qT)^4+O\left(2^{\omega(q)}T(\log qT)^3(\log\log3q)^5\r).\notag
\end{align}
In 2010, Bui and Heath-Brown \cite{BH10} proved the leading term for all $q\not\equiv 2 \pmod 4$ and $T\ge2$ by sharpening the error term in \eqref{eqRan}. To be specific, they proved that, for $q\not\equiv 2 \pmod 4$ and $T\ge 2$,
\begin{align}
&\frac1{\varphi^*(q)}\mathop{\sum\nolimits^*}_{\chi ~(\bmod q)}\int_0^T\left|L\left(\tfrac12+it,\chi\r)\r|^4dt\notag\\
&\ \ =\left(1+O\left(\frac{\omega(q)}{\log q}\sqrt{\frac{q}{\varphi(q)}}\r)\r)\frac{T}{2\pi^2}\prod_{p\mid q}\frac{(1-p^{-1})^3}{(1+p^{-1})}(\log qT)^4+O\left(\frac{q}{\varphi^*(q)}T(\log qT)^{\frac72}\r).\notag
\end{align}

After applying the approximate functional equation (see also Lemma \ref{lemafe}), the leading term comes from the diagonal terms. But to distinguish other main terms, one should deduce an asymptotic formula for the off-diagonal terms. By extending the method of Heath-Brown \cite{HB79}, Wang \cite{Wan88} tried to distinguish all main terms, proving that
\begin{align}
  \frac1{\varphi^*(q)}\mathop{\sum\nolimits^*}_{\chi ~(\bmod q)}\int_0^T&\left|L\left(\tfrac12+it,\chi\r)\r|^4dt=T\sum_{j=0}^4a_j(\log qT)^j+O\left( \frac{q}{\varphi^*(q)}\min\left\{q^{\frac18}T^{\frac78+\ve},T^{\frac{11}{12}+\ve}\r\}\r)\notag
\end{align}
with
\[
a_4=\frac1{2\pi^2}\prod_{p\mid q}\frac{(1-p^{-1})^{3}}{(1+p^{-1})}, \ \ \ \ \ \ \ \ a_j\ll q^{\ve}\ \ \ \text{for} \ \ \ j=0,1,2,3.
\]
 This asymptotic formula is non-trivial only for large $T\gg q^{1+\ve}$, and it is hardly to distinguish an explicit dependence of the coefficients $a_j$ on $q$, for $0\le j\le 3$.

This work is devoted to deducing the asymptotic formula in Conjecture \ref{conjT}, with an error term owning a powering saving from $q$-aspect and $t$-aspect simultaneously, so that the asymptotic formula would hold uniformly in all $T$ and $q$.
\begin{theorem}\label{thmmain}
We have that Conjecture \ref{conjT} holds but with an error term $\mathscr{E}(T,q)$, bounded uniformly in $T$ and $q$ that
\begin{align}\label{eqthmmainE}
\mathscr{E}(T,q)\ll T_1^{1-\frac{1-6\theta}{382-96\theta}+\ve}q^{-\frac{1-6\theta}{382-96\theta}+\ve}.
\end{align}
Moreover, we have
\[
\mathscr{E}(T,q)\ll T_1^{1+\ve}q^\ve\Delta,
\]
where we may take $\Delta$ freely among
\begin{align}\label{Delta1}
T_1^{\frac{11}7}q^{-\frac1{14}+\frac37\theta}\ \ \ \  \text{and}\ \ \ \ T_1^{-\frac1{16}}.
\end{align}
\end{theorem}

The bound in \eqref{eqthmmainE} is just a special form to gain the same power saving from both $t$-aspect and $q$-aspect, and it is obvious a direct result of \eqref{Delta1},
For prime moduli, we can have a much better bound on the error term.
\begin{theorem}\label{thmmainp}
For prime $p\ge3$, we have that Conjecture \ref{conjT} holds but with an error term $\mathscr{E}(T,p)$, bounded uniformly in $T$ and $p$ that
\begin{align}\label{eqthmmainEp}
\mathscr{E}(T,p)\ll T_1^{1-\frac1{60}+\ve}p^{-\frac{1}{60}+\ve}.
\end{align}
Moreover, we have
\[
\mathscr{E}(T,p)\ll T_1^{1+\ve}p^\ve\Delta_1,
\]
where we may take $\Delta_1$ freely among
\begin{align}\label{Delta2}
T_1^{\frac{11}7}p^{-\frac1{14}}\ \ \ \  \text{and}\ \ \ \ \max\left\{T_1^{-\frac12},~T_1^{-\frac1{16}}p^{-\frac1{16}}\right\}.
\end{align}
\end{theorem}

When $T$ is small, considering the moment with a smooth function on $t$ could make the average essentially easy, as well as a considerable power saving from $q$-aspect in the error term. However, this hardly has any application here. Actually, the cost of last removing the smooth function will be too large to reserve any saving from $q$-aspect if $T$ is small with respect to $q$. A more feasible way is to extend the treatment at the central point by regarding $t$ as a parameter. Since the power saving from $q$-aspect is small at the central point, there is little room for expenditure in $q$-aspect when we treat the large $T$ case. That is to say, to cover all the range, it is important to gain power saving from $t$-aspect while costing nothing in $q$-aspect.

\subsection{Sketch of the proof of Theorems \ref{thmmain} and \ref{thmmainp}}

We split the averaging over $t$ into two parts, according  to the size of $t$ that $t\le q^{\ve_0}$ and $t>q^{\ve_0}$ for some small $\ve_0>0$, and handle them  with different methods.

For small $t$, the cost of applying a smooth function is large, and it is hardly to expect any remarkable power saving from the averaging over $t$. Thus, we treat the integrand directly, and pay our main attention to the saving from $q$-aspect. Thinking of $t$ as a parameter, we may extend the treatment at the central point in \cite{Wu20} to get the following result.
\begin{theorem}\label{corLq}
For $q\not\equiv 2 \pmod 4$ and $0\le t\asymp T$, we have
\begin{align}\label{eqsmallt}
\frac1{\varphi^*(q)}&\mathop{{\sum}^*}_{\chi~(\bmod q)}\left|L\left(\tfrac12+it,\chi\right)\right|^4=\prod_{p\mid q}\frac{(1-p^{-1})^3}{(1+p^{-1})}\\
&\times\sum_{j=0}^4\frac{c_j}2\sum_{\ma=0,1} \left(\log\frac{q}{\pi} +\frac12 \frac{\Gamma'}{\Gamma}\left(\tfrac{\frac12-it+\ma}2\r) +\frac12\frac{\Gamma'}{\Gamma}\left(\tfrac{\frac12+it+\ma}2\r)\r)^j +O\left(T_1^{\frac{11}7+\ve} q^{-\frac1{14}+\frac37\theta+\ve}\r),\notag
\end{align}
where $\theta$ denotes the exponent towards the Ramanujan--Petersson conjecture.
\end{theorem}

The error term in \eqref{eqsmallt} is non-trivial for $T\ll q^{\frac18-\frac34\theta}$, but it is weaker than the main terms only for $T\ll q^{\frac1{22}-\frac3{11}\theta}$. The best known value of $\theta$ is $7/64$, proved by  Kim and Sarnak~\cite{Kim03}. Thanks to Blomer, Fouvry,  Kowalski, Michel and Mili{\'c}evi{\'c}~\cite{BFK+17a,BFK+17b}, we may remove the dependence on the
Ramanujan-Petersson conjecture and take $\theta=0$ for prime moduli .

When $t$ is large, we appeal to a weighted function $\Phi(t)$ to force $m$ and $n$ to be close to each other.  After some technical treatments, we may transform the problem essentially to a quadratic divisor problem
\begin{align}\label{eqqdp}
\frac1{\varphi^*(q)}\sum_{d\mid q}\vp(d)\mu\left(\frac qd\r)\sum_{\substack{m_1m_2-n_1n_2=\pm h\neq0 \\ (m_1m_2n_1n_2,q)=1 \\ d\mid h}} F\left(\frac hH,\frac{m_1}{M_1},\frac{m_2}{M_2}, \frac{n_1}{N_1},\frac{n_2}{N_2}\r)
\end{align}
for a compact support function $F$. A divisor problem as in \eqref{eqqdp} but without the coprime condition $(m_1m_2n_1n_2,q)=1$ has been well studied. By the delta method, Duke, Friedlander, and Iwaniec \cite{DFI94} provided an asymptotic formula for a remarkable range of $h$.
Bettin, Bui, Li, and Radziwi\l\l \cite{BBLR16} introduced a different way to treat the divisor problem, which works specially for small $h$ and provides a sharp error term. We extend the way to adapt the coprime condition. We will obtain an asymptotic formula for \eqref{eqqdp} with a remarkable power saving from $t$-aspect but not costing $q$-aspect.
The treatment would also borrow some technologies from Young \cite{You10} and Bettin, Chandee and Radziwi\l\l \cite{BCR17}, etc.

\begin{theorem}\label{thmmainT}
For $q\not\equiv 2 \pmod 4$ and $T\gg q^\ve$, we have
\begin{align}\label{eqT2T}
\frac1{\varphi^*(q)}&\mathop{{\sum}^*}_{\chi~(\bmod q)}\int_{T}^{2T}\left|L\left(\tfrac12+it,\chi\right)\right|^4dt\\
=&\prod_{p\mid q}\frac{(1-p^{-1})^3}{(1+p^{-1})}\sum_{j=0}^4c_j \int_T^{2T}\left(\log\frac{tq}{2\pi}\right)^jdt+O\left(\max\left\{T^{\frac{15}{16}}(q/q_0^2)^{-\frac1{16}}, T^{\frac12}\right\}T^\ve q^\ve\r),\notag
\end{align}
where $q_0=\max\{d: d\mid q^*, d<{q^*}^{\frac12}\}$ with $q^*=\prod_{p\mid q}p$.
\end{theorem}

The error term in \eqref{eqT2T} provides a power saving not less than $T^{-\frac 1{16}+\ve}q^\ve$ for all moduli, in particular, a power saving $T^{-\frac 1{16}+\ve}q^{-\frac 1{16}+\ve}$ for prime moduli with large $T$.

Theorems \ref{thmmain} and \ref{thmmainp} are direct results of Theorems \ref{corLq} and \ref{thmmainT}, and Theorems \ref{corLq} and \ref{thmmainT} are based on two shifted moments in next section.

\subsection{Two shifted moments}
Let $\Phi(t)$ be a smooth, nonnegative function with support contained in $[T/2,4T]$, satisfying $\Phi^{(j)}(t)\ll_j T_0^{-j}$ for all $j=0,1,2,\ldots$, where $T^{\frac12+\ve}\ll T_0\ll T$. We have chosen to compute two shifted fourth moments of Dirichlet $L$-functions, which include the parameters $\alpha,~\beta,~\gamma,~\delta$ or a weighted function $\Phi(t)$, and doing so allows for a clearer structure of the main terms. The first one is given by
\begin{align}
M(\alpha,\beta,\gamma,&\delta,t)=\frac1{\varphi^*(q)}\mathop{{\sum}^*}_{\chi~(\bmod q)}\notag\\
&\times L\left(\tfrac12+it+\alpha,\chi\r) L\left(\tfrac12+it+\beta,\chi\r)L\left(\tfrac12-it+\gamma,\ol{\chi}\r) L\left(\tfrac12-it+\delta,\ol{\chi}\r).\notag
\end{align}
This moment does not contain the averaging over $t$, and we consider its asymptotic formula when $t$ is small.
The second shifted moment is defined via
\begin{align}
M(\alpha,\beta,\gamma,&\delta,\Phi)=\frac1{\varphi^*(q)}\mathop{{\sum}^*}_{\chi~(\bmod q)}\int_{\mathbb{R}}\Phi(t)\notag\\
&\times L\left(\tfrac12+it+\alpha,\chi\r) L\left(\tfrac12+it+\beta,\chi\r)L\left(\tfrac12-it+\gamma,\ol{\chi}\r) L\left(\tfrac12-it+\delta,\ol{\chi}\r)dt.\notag
\end{align}

To present asymptotic formulae for these two shifted moments, we should introduce some notations for convenience. Let
\begin{align}\label{eqdefZ}
Z_{q}(\alpha,\beta,\gamma,\delta)=\frac{\zeta_q(1+\alpha+\gamma) \zeta_q(1+\alpha+\delta)\zeta_q(1+\beta+\gamma)\zeta_q(1+\beta+\delta)} {\zeta_q(2+\alpha+\beta+\gamma+\delta)},
\end{align}
\begin{align}
X_{\alpha,\gamma}(q,t,\ma)=\left(\frac q{\pi}\r)^{-\alpha-\gamma} \frac{\Gamma\left(\frac{\frac12-\alpha-it+\ma}2\r)} {\Gamma\left(\frac{\frac12+\alpha+it+\ma}2\r)} \frac{\Gamma\left(\frac{\frac12-\gamma+it+\ma}2\r)} {\Gamma\left(\frac{\frac12+\gamma-it+\ma}2\r)},
\end{align}
and
\begin{align}\label{eqX}
X_{\alpha,\beta,\gamma,\delta}(q,t,\ma)=X_{\alpha,\gamma}(q,t,\ma)X_{\beta,\delta}(q,t,\ma)
\end{align}
with $\ma=0,1$.
Obviously, $Z_{q}(\alpha,\beta,\gamma,\delta)$ and $X_{\alpha,\beta,\gamma,\delta}(q,t,\ma)$ are symmetric with respect to the parameters $\alpha, \beta$ and also symmetric with respect to $\gamma, \delta$.

\begin{theorem}\label{thmLq}
For $q\not\equiv 2 \pmod 4$, $t\asymp T\ge 0$, and $\alpha, \beta, \gamma, \delta\ll(\log T_1q)^{-1}$, we have
\begin{align}
M(\alpha,\beta,\gamma,\delta, &t)=Z_{q}(\alpha,\beta,\gamma,\delta) +Z_{q}(-\gamma,-\delta,-\alpha,-\beta) \left(\frac12\sum_{\ma=0,1}X_{\alpha,\beta,\gamma,\delta}(q,t,\ma)\right)\notag\\
&+Z_{q}(\beta,-\gamma,\delta,-\alpha)\left(\frac12\sum_{\ma=0,1}X_{\alpha,\gamma}(q,t,\ma)\right) +Z_{q}(\alpha,-\gamma,\delta,-\beta)\left(\frac12\sum_{\ma=0,1}X_{\beta,\gamma}(q,t,\ma)\right)\notag\\
&+Z_{q}(\beta,-\delta,\gamma,-\alpha) \left(\frac12\sum_{\ma=0,1}X_{\alpha,\delta}(q,t,\ma)\right) +Z_{q}(\alpha,-\delta,\gamma,-\beta)\left(\frac12\sum_{\ma=0,1}X_{\beta,\delta}(q,t,\ma)\right)\notag\\
& +O\left(T_1^{\frac{11}7+\ve} q^{-\frac1{14}+\frac37\theta+\ve}\r).\notag
\end{align}
\end{theorem}
This theorem can be seen as an extension of Theorem 1.3 in \cite{Wu20}.
By regarding $t$ as a parameter, we may deduce Theorem \ref{thmLq} following the treatment of \cite[Theorem 1.3]{Wu20} step by step. Differences will come from the ratio of gamma factors, as well as the factor $\left(\frac mn\r)^{-it}$ in the asymptotic functional equation (see Lemma \ref{lemafe}). These differences will not bring any essential changes in calculating the main terms, and will contribute at most a factor $T_1^{\frac{11}7+\ve}$ to the error term, which we will specify in Section \ref{secthmlq}.

\begin{theorem}\label{thmLT}
For $q\neq 2 \pmod 4$, $T\gg q^\ve$, and $\alpha, \beta, \gamma, \delta\ll(\log T_1q)^{-1}$, we have
\begin{align}
M(\alpha,\beta,\gamma,\delta,\Phi)&=Z_{q}(\alpha,\beta,\gamma,\delta) \int_R \Phi(t)dt +Z_{q}(-\gamma,-\delta,-\alpha,-\beta) \int_R \Phi(t)\left(\frac{tq}{2\pi}\r)^{-\alpha-\beta-\gamma-\delta}dt\notag\\
&+Z_{q}(\beta,-\gamma,\delta,-\alpha) \int_R \Phi(t)\left(\frac{tq}{2\pi}\r)^{-\alpha-\gamma}dt +Z_{q}(\alpha,-\gamma,\delta,-\beta) \int_R \Phi(t)\left(\frac{tq}{2\pi}\r)^{-\beta-\gamma}dt\notag\\
&+Z_{q}(\beta,-\delta,\gamma,-\alpha) \int_R \Phi(t)\left(\frac{tq}{2\pi}\r)^{-\alpha-\delta}dt +Z_{q}(\alpha,-\delta,\gamma,-\beta) \int_R \Phi(t)\left(\frac{tq}{2\pi}\r)^{-\beta-\delta}dt\notag\\
&+{O\left( T^{\frac34+\ve}(q/q_0^2)^{-\frac14+\ve}(T/T_0)^3+T^\ve q^\ve(T/T_0)\r)},\notag
\end{align}
where $q_0=\max\{d: d\mid q^*, d<{q^*}^{\frac12}\}$ with $q^*=\prod_{p\mid q}p$.
\end{theorem}

\subsection{Proof of Theorem \ref{corLq} and Theorem \ref{thmmainT} from the shifted moments}
For main terms of the asymptotic formulae in Theorems \ref{thmLq} and \ref{thmLT}, the symmetry implies that all poles cancel out to form the holomorphy with respect to the shift parameters, which has been proved in a more general setting in Lemma 2.5.5 of \cite{CFK+05}. Thus, taking the limit as all shifts go to zero in Theorem \ref{thmLq} gives Theorem \ref{corLq}.

The proof of Theorem \ref{thmmainT} needs some narrative, but it is standard. Actually, we would obtain Theorem \ref{thmmainT}
by taking appropriate weighted functions $\Phi(t)$ in Theorem \ref{thmLT}.
Let $0\le\Phi_1(t)\le1$ be a weighted function supported on $[T, 2T]$, which is identical to unity when $T+T_0^{1+\ve}\le t\le 2T-T_0^{1+\ve}$; let $0\le\Phi_2(t)\le 1$ be supported on $[T-T_0^{1+\ve},2T+T_0^{1+\ve}]$, which is identical to unity when $T\le t\le 2T$.
It is obvious that
\begin{align}\label{eqintT}
M(0,0,0,0,\Phi_1)\le\frac1{\varphi^*(q)}\mathop{{\sum}^*}_{\chi~(\bmod q)}\int_{T}^{2T}\left|L\left(\tfrac12+it,\chi\right)\right|^4dt\le M(0,0,0,0,\Phi_2).
\end{align}
On the other hand, taking the limit as all shifts go to zero in Theorem \ref{thmLT} shows that, for $i=1,2$,
\begin{align}
M(0,0,0,0,\Phi_i)=\prod_{p\mid q}\frac{(1-p^{-1})^3}{(1+p^{-1})}&\sum_{j=0}^4c_j \int_T^{2T}\left(\log\frac{tq}{2\pi}\right)^j dt\notag\\
&+O\left( T^{\frac34+\ve}(q/q_0^2)^{-\frac14+\ve}(T/T_0)^3+T_0^{1+\ve}q^\ve\r).\notag
\end{align}
Inserting this into \eqref{eqintT} and taking $T_0=\max\left\{T^{\frac{15}{16}}(q/q_0^2)^{-\frac1{16}}, T^{\frac12}\right\}$, we would establish Theorem \ref{thmmainT}.

The remainder is devoted to proving Theorems \ref{thmLq} and \ref{thmLT}, where we may impose some restrictions on the shifts.
More precisely, we assume that each of the shifts lies in a fixed annulus with inner and outer radii $\asymp(\log T_1q)^{-1}$, which are separated enough so that $|\alpha\pm\beta|\gg(\log T_1q)^{-1}$, etc. We initially prove the theorems with these restrictions in place. Since every terms in the asymptotic formulae are holomorphic, the maximum modulus principle extends our results to all shifts $\ll(\log T_1q)^{-1}$.

\begin{notation}
We use the common convention that $\ve$ denotes an arbitrarily small positive constant which may vary from line to line, and that notations $(a,b)$, $[a,b]$ are the gcd and lcm of $a$ and $b$ respectively. The notation $\sigma_{\alpha,\beta}(n)$ is defined via
\[
\sigma_{\alpha,\beta}=\sum_{d_1d_2=n}d_1^\alpha d_2^\beta.
\]
\end{notation}

\section{Background and auxiliary lemmas}

\subsection{Dirichlet $L$-functions}
Let $q$ be a positive integer and $\chi$ be a primitive character modulo $q$. The Dirichlet $L$-function $L(s,\chi)$ is defined as
\begin{align}
  L(s,\chi)=\sum_n\chi(n)n^{-s}\notag
\end{align}
for $\re(s)>1$. Let
\begin{align}
\ma=\left\{
\begin{aligned}
&0,\ \ \text{for} \ \ \chi(-1)=1,\\
&1,\ \ \text{for} \ \ \chi(-1)=-1,\notag
\end{aligned}
\right.
\end{align}
and let
\begin{align}
  \Lambda(s,\chi)=\left(\frac q{\pi}\r)^{\frac s2}\Gamma\left(\frac{s+\ma}{2}\r)L(s,\chi).\notag
\end{align}
After extended to the whole plane, the Dirichlet $L$-function satisfies the following functional equation
\begin{align}
  \Lambda(s,\chi)=i^{-\ma}q^{-\frac12}\tau(\chi)\Lambda(1-s,\ol{\chi}) \ \ \ \ \text{with}\ \ \ \tau(\chi)=\sum_{n~(\bmod{q})}\chi(n)e\left(\frac nq\right).
\end{align}

\subsection{Approximate functional equation}
\begin{lemma}[Approximate functional equation]
\label{lemafe}
Let $G(s)$ be an even entire function of exponential decay in any strip $|\re(s)|<C$, satisfying $G(0)=1$. For $x>0$ and $\ma=0,1$, we define
\begin{align}
\label{defV}
V_{\alpha,\beta,\gamma,\delta}(x,t,\ma)=\frac1{2\pi i}\int_{(1)}\frac{G(s)}{s}g_{\alpha,\beta,\gamma,\delta}(s,t,\ma)x^{-s}ds,
\end{align}
where
\begin{align}
\label{defg}
g_{\alpha,\beta,\gamma,\delta}(s,t,\ma)=\pi^{-2s}\frac{\Gamma\left(\frac{ \frac12+\alpha+s+it+\ma}2\r) \Gamma\left(\frac{ \frac12+\beta+s+it+\ma}2\r) \Gamma\left(\frac{ \frac12+\gamma+s-it+\ma}2\r) \Gamma\left(\frac{ \frac12+\delta+s-it+\ma}2\r)}{\Gamma\left(\frac{\frac12+\alpha+it+\ma}2\r) \Gamma\left(\frac{\frac12+\beta+it+\ma}2\r) \Gamma\left(\frac{\frac12+\gamma-it+\ma}2\r) \Gamma\left(\frac{\frac12+\delta-it+\ma}2\r)}.
\end{align}
Furthermore, let
\begin{align}\label{eqdefwtV}
\widetilde{V}_{\alpha,\beta,\gamma,\delta}(x,t,\ma)= X_{-\gamma,-\delta,-\alpha,-\beta}(q,t,\ma) V_{\alpha,\beta,\gamma,\delta}(x,t,\ma)
\end{align}
with $X_{-\gamma,-\delta,-\alpha,-\beta}(q,t,\ma)$ being defined as in \eqref{eqX}.
Then, for $\chi(-1)=(-1)^\ma$, we have
\begin{align}
L&\left(\tfrac12+it+\alpha,\chi\r) L\left(\tfrac12+it+\beta,\chi\r) L\left(\tfrac12-it+\gamma,\ol{\chi}\r) L\left(\tfrac12-it+\delta,\ol{\chi}\r)\\
&=\sum_{m,n}\frac{\sigma_{\alpha,\beta}(m)\sigma_{\gamma,\delta}(n) \chi(m)\ol{\chi}(n)} {(mn)^{\frac12}}\left(\frac mn\r)^{-it}V_{\alpha,\beta,\gamma,\delta}\left(\frac{mn}{q^2},t,\ma\r)\notag\\ &\ \ \ + \sum_{m,n}\frac{\sigma_{-\gamma,-\delta}(m)\sigma_{-\alpha,-\beta}(n) \chi(m)\ol{\chi}(n)} {(mn)^{\frac12}}\left(\frac mn\r)^{-it}\wt{V}_{-\gamma,-\delta,-\alpha,-\beta}\left(\frac{mn}{q^2},t,\ma\r).\notag
\end{align}
\end{lemma}

This approximate functional equation can be deduced standardly from the function equation of $L(s,\chi)$; see also \cite{HY08} and \cite[Proposition 2.4]{You11}. The approximate functional equation holds for a general $G$, and we will appeal to a special one.

\begin{definition}[Definition of $G(s)$]\label{DefG}
Let $G(s)=P_{\alpha,\beta,\gamma,\delta}(s)\exp(s^2)$, where $P_{\alpha,\beta,\gamma,\delta}(s)$ is a even polynomial in $s$ satisfying the following common properties: it takes the value $1$ at $s=0$; it is rational in the shifts $\alpha,\beta,\gamma,\delta$; it is symmetric in the shifts; it is invariant under $\alpha\rightarrow-\alpha$, $\beta\rightarrow-\beta$, etc.; it also takes zero at $s=-\frac{\alpha+\gamma}2$ (as well as other points by symmetry).
\end{definition}
\subsection{Results due to Stirling's approximation}
We present some results about the ratios of gamma functions arising in the approximate functional equation.
\begin{lemma}
\label{lemV}
For $t$ large, we have
\begin{align}\label{lemVX}
X_{\alpha,\beta,\gamma,\delta}(q,t,\ma)=\left(\frac {tq}{2\pi}\r)^{-\alpha-\beta-\gamma-\delta}\left(1+O(t^{-1})\r),
\end{align}
and for $j\ge0$,
\begin{align}\label{lemVX1}
\frac{\partial^j}{\partial t^j}X_{\alpha,\beta,\gamma,\delta}(q,t,\ma)\ll_{j}t^{-j}.
\end{align}
\end{lemma}

\begin{lemma}\label{lemDV}
For $t$ large and $s$ in any fixed vertical strip, we have
\begin{align}\label{lemVg}
g_{\alpha,\beta,\gamma,\delta}(s,t,\ma)=\left(\frac t{2\pi}\r)^{2s} \left(1+O\left(t^{-1}(1+|s|^2)\r)\r).
\end{align}
Moreover, we have
\begin{align}\label{lemVV}
t^j\frac{\partial^j}{\partial t^j}V_{\alpha,\beta,\gamma,\delta}(x,t,\ma)\ll_{A,j}\left(1+|x|/t^2\r)^{-A}
\end{align}
for any fixed $A>0$ and $j\ge0$.
\end{lemma}

These two lemmas are well-known results, deduced from Stirling's approximation standardly. To eliminate the difference between even and odd characters in the sum of $M(\alpha,\beta,\gamma,\delta,\Phi)$, we appeal to the following lemma.

\begin{lemma}
For $t$ large, we have
\begin{align}
\label{lemV2}
&V_{\alpha,\beta,\gamma,\delta}(x,t,0)-V_{\alpha,\beta,\gamma,\delta}(x,t,1) \ll  t^{-1+\ve},\\
\label{lemV3}
&\wt{V}_{\alpha,\beta,\gamma,\delta}(x,t,0)-\wt{V}_{\alpha,\beta,\gamma,\delta}(x,t,1) \ll  t^{-1+\ve}.
\end{align}

\end{lemma}
\begin{proof}
Due to \eqref{lemVV}, we assume $x\ll t^{2+\ve}$ in \eqref{lemV2} and \eqref{lemV3} since the estimates are obvious otherwise.
Recalling the definition of $V$ in \eqref{defV}, we rewrite that
\begin{align}
V_{\alpha,\beta,\gamma,\delta}(x,t,0)-V_{\alpha,\beta,\gamma,\delta}(x,t,1)=\frac1{2\pi i}\int_{(1)}\frac{G(s)}{s}\left(g_{\alpha,\beta,\gamma,\delta}(s,t,0) -g_{\alpha,\beta,\gamma,\delta}(s,t,1)\right)x^{-s}ds.\notag
\end{align}
We move the integral to $\re(s)=-\ve$ without encountering any poles, observing that $g_{\alpha,\beta,\gamma,\delta}(s,t,0)-g_{\alpha,\beta,\gamma,\delta}(s,t,1)$ takes zeros at $s=0$. On the new path, we can easily see \eqref{lemV2} from a result of \eqref{lemVg} that
\begin{align}
  g_{\alpha,\beta,\gamma,\delta}(s,t,0)-g_{\alpha,\beta,\gamma,\delta}(s,t,1)\ll t^{-1}(1+|s|^2).\notag
\end{align}

By the definition of $\wt{V}$ in \eqref{eqdefwtV}, the estimate \eqref{lemV3} is an immediate result of \eqref{lemVX}, \eqref{lemVV} and \eqref{lemV2}. This establishes the lemma.
\end{proof}

\subsection{The number of primitive characters}
Let $\varphi^*(q)$ denote the number of primitive characters modulo $q$. It is known that $\varphi^*(q)$ is a multiplicative function defined by
\begin{align}
  \varphi(p^m)=\left\{
  \begin{aligned}
  &p^{m-2}(p-1)^2, \ \ \text{for}\ \ m\ge2,\\
  &p-2, \ \ \ \ \ \ \ \ \ \ \ \ \ \ \text{for}\ \ m=1.\notag
  \end{aligned}
  \right.
\end{align}

\subsection{The orthogonality formula}
\begin{lemma}[The orthogonality formula]
\label{lemof}
For $(mn,q)=1$, we have
\begin{align}
\label{lemof1}
\mathop{{\sum}^*}_{\chi~(\bmod q)}\chi(m)\ol{\chi}(n)=\sum_{d\mid(q,m-n)}\vp(d)\mu(q/d).
\end{align}
Moreover
\begin{align}
\label{lemof2}
\mathop{{\sum}^*}_{\substack{\chi~(\bmod q)\\ \chi(-1)=(-1)^\ma}}\chi(m)\ol{\chi}(n) =\frac12\sum_{d\mid(q,m-n)}\vp(d)\mu(q/d) +\frac{(-1)^\ma}2\sum_{d\mid(q,m+n)}\vp(d)\mu(q/d).
\end{align}
\end{lemma}
This orthogonality formula is well-known, and its proof may be refereed to \cite{HB81} and \cite{Sou07}.

\subsection{Two partitions of unity}\label{secpartition}
We appeal to a partition introduced in \cite{BBLR16}. Let $f$ be a smooth function that
\begin{align}
f(x)+f(1/x)=1\notag
\end{align}
for all $x\in \mathbb{R}$ and $f(x)\ll_j(1+x)^{-j}$ for any fixed $j>0$ and $x>1$. Also it has the Mellin inversion
\begin{align}
f(x)=\frac1{2\pi i}\int_{(\ve)}\widehat{f}(u)x^{-u}du,\notag
\end{align}
where $\widehat{f}(u)$ has a simple pole at $u=0$ with residue $1$, and satisfies
\begin{align}
\widehat{f}\left(\pm \frac{(\alpha-\beta)}2\r)=\widehat{f}\left(\pm \frac{(\gamma-\delta)}2\r)=0.\notag
\end{align}
One may apply the identity
\begin{align}\label{efg}
f\left(\frac{m_1}{m_2}\r)f\left(\frac{n_1}{n_2}\r)
+f\left(\frac{m_2}{m_1}\r)f\left(\frac{n_1}{n_2}\r) +f\left(\frac{m_1}{m_2}\r)f\left(\frac{n_2}{n_1}\r) +f\left(\frac{m_2}{m_1}\r)f\left(\frac{n_2}{n_1}\r)=1
\end{align}
to partition unity into four roughly similar terms. Then in each term, there exists a comparison on the sizes of $m_1, m_2$ and $n_1, n_2$.

The second one is the dyadic partition. Let $W(x)$ be a smooth non-negative function compactly supported on $[1,2]$ such that
\begin{align}
\sum_{M}W\left(\frac xM\r)=1,\notag
\end{align}
where $M$ varies over a set of positive real numbers with $\#\{M: X^{-1}\le M\le X\}\ll(\log X)$. The $W$ function has the Mellin pair
\begin{align}
\left\{
\begin{aligned}
&\widehat{W}(u)=\int_0^\infty W(x)x^{u_1-1}dx,\notag\\
&W(x)=\frac1{2\pi i}\int_{(c_u)}\widehat{W}(u)x^{-u}du.
\end{aligned}
\right.
\end{align}

\section{Initial treatment of the shifted moment}\label{secT}
From this section, we start our proof of Theorem \ref{thmLT}, which occupies next two sections. We assume $T\gg q^\ve$, a convention that holds throughout the proof of Theorem \ref{thmLT}.

\subsection{Initial treatment}
Using the approximate functional equation stated in Lemma \ref{lemafe}, we break $M(\alpha,\beta,\gamma,\delta,\Phi)$ into two terms that
\begin{align}\label{eMD}
M(\alpha,\beta,\gamma,\delta,\Phi)=A_1(\alpha,\beta,\gamma,\delta,\Phi) +A_{-1}(-\gamma,-\delta,-\alpha,-\beta,\Phi),
\end{align}
where $A_1$ is the contribution from the `first part' of the approximate functional equation that
\begin{align}
A_1=&\frac1{\varphi^*(q)}\mathop{{\sum}^*}_{\chi~(\bmod q)}\sum_{m,n}\frac{\sigma_{\alpha,\beta}(m) \sigma_{\gamma,\delta}(n)\chi(m)\overline{\chi}(n)} {(mn)^{\frac12}}\int_R\left(\frac mn\r)^{-it}
V_{\alpha,\beta,\gamma,\delta}\left(\frac{mn}{q^2},t,\ma\r)\Phi(t)dt,\notag
\end{align}
and $A_{-1}$ is the `second part' that
\begin{align}
A_{-1}=&\frac1{\varphi^*(q)}\mathop{{\sum}^*}_{\chi~(\bmod q)}\sum_{m,n}\frac{\sigma_{-\gamma,-\delta}(m)\sigma_{-\alpha,-\beta}(n) \chi(m)\overline{\chi}(n)} {(mn)^{\frac12}}\int_R\left(\frac mn\r)^{-it}
V_{-\gamma,-\delta,-\alpha,-\beta}\left(\frac{mn}{q^2},t,\ma\r)\Phi(t)dt\notag
\end{align}
Our major focus is on the evaluation of $A_1$, and the treatment of $A_{-1}$ is identical.

Before applying the orthogonality formula of primitive characters, we first remove the dependence of $V_{\alpha,\beta,\gamma,\delta}(x,t,\ma)$ on the parity of $\chi$ by rewriting it into two parts as
\begin{align}
\label{eqpv}
V_{\alpha,\beta,\gamma,\delta}(x,t,\ma) =&\frac12\left(V_{\alpha,\beta,\gamma,\delta}(x,t,0) +V_{\alpha,\beta,\gamma,\delta}(x,t,1)\r)\\ &+\frac{\chi(-1)}{2}\left(V_{\alpha,\beta,\gamma,\delta}(x,t,0) -V_{\alpha,\beta,\gamma,\delta}(x,t,1)\r).\notag
\end{align}
Note that the second part would just contribute an  error to $A_1$. To be specific, we insert \eqref{eqpv} into $A_1$, and then the contribution of the second part is
\begin{align}\label{eqsecondpart}
\frac1{2\varphi^*(q)} \mathop{{\sum}^*}_{\chi~(\bmod q)}\sum_{m,n} &\frac{\sigma_{\alpha,\beta}(m)\sigma_{\gamma,\delta} (n) \chi(-m)\ol{\chi}(n)}{(mn)^{\frac12}}\\
&\times\int_R \left(\frac mn\r)^{-it}\left(V_{\alpha,\beta,\gamma,\delta}\left(\frac{mn}{q^2},t,0\r) -V_{\alpha,\beta,\gamma,\delta}\left(\frac{mn}{q^2},t,1\r)\r)\Phi(t)dt.\notag
\end{align}
The averaging over $t$-aspect forces $m$ and $n$ to be very close to each other. More precisely, integration by parts shows
\begin{align}
\int_R \left(\frac mn\r)^{-it}\left(V_{\alpha,\beta,\gamma,\delta}\left(\frac{mn}{q^2},t,0\r) -V_{\alpha,\beta,\gamma,\delta}\left(\frac{mn}{q^2},t,1\r)\r)\Phi(t)dt\ll_j \frac T{(T_0\log\frac mn)^j}\notag
\end{align}
for any $j\ge1$, which yields that the integral over $t$ is very small unless $\left|1-\frac{m}{n}\right|\ll T_0^{-1+\ve}$.
After applying the orthogonality formula \eqref{lemof1} and the estimate \eqref{lemV2}, we find that \eqref{eqsecondpart} is bounded by
\begin{align}
&\ll \frac{T^{\ve}}{\varphi^*(q)}\sum_{d\mid (q,m+n)}\vp(d)\sum_{\substack {mn\le (Tq)^{2+\ve}\\ \left|1-\frac{m}{n}\right|\ll T_0^{-1+\ve}}}\frac{\sigma_{\alpha,\beta}(m)\sigma_{\gamma,\delta}(n)} {(mn)^{\frac12}}+O(T^{-2020}q^{-2020})\notag\\
&\ll\frac{T^{\ve}q^\ve}{\varphi^*(q)}\sum_{d\mid q}\vp(d)\frac{Tq}{T_0^{1-\ve}\varphi(d)}\ll T^\ve q^{\ve}(T/T_0),\notag
\end{align}
which is the second error term of the asymptotic formula in Theorem \ref{thmLT}.

Note that the first part of \eqref{eqpv} has nothing to do with the parity of $\chi$, and we can just average all primitive characters in $A_1$ to evaluate its contribution.
After applying the orthogonality formula to this part, we find that
\begin{align}\label{eq+A1}
A_1(\alpha,\beta,\gamma,\delta,\Phi)=&\frac12\sum_{\ma=0,1}\frac1{\varphi^*(q)} \sum_{d\mid q}\vp(d)\mu\left(\frac qd\r) \sum_{\substack{(mn,q)=1\\ m\equiv n~(\bmod d)}}\frac{\sigma_{\alpha,\beta}(m) \sigma_{\gamma,\delta}(n)} {(mn)^{\frac12}}\\
&\times\int_R\left(\frac mn\r)^{-it}V_{\alpha,\beta,\gamma,\delta}\left(\frac{mn}{q^2},t,\ma\r) \Phi(t)dt+O(T^\ve q^{\ve}(T/T_0) ).\notag
\end{align}
It is easy to see that a similar expression holds for $A_{-1}$.

We break the sum in \eqref{eq+A1} into diagonal terms and off-diagonal terms, that is
\begin{align}\label{eA1D}
A_1(\alpha,\beta,\gamma,\delta,\Phi)=A_D(\alpha,\beta,\gamma,\delta,\Phi) +A_O(\alpha,\beta,\gamma,\delta,\Phi)+O(T^\ve q^{\ve}(T/T_0) ),
\end{align}
where
\begin{align}\label{eq+AD}
A_D(\alpha,\beta,\gamma,\delta,\Phi)=&\frac12\sum_{\ma=0,1} \frac1{\varphi^*(q)}\sum_{d\mid q}\vp(d)\mu\left(\frac qd\r) \sum_{(n,q)=1}\frac{\sigma_{\alpha,\beta}(n)\sigma_{\gamma,\delta}(n)} {n}\\
&\times\int_R V_{\alpha,\beta,\gamma,\delta}\left(\frac{n^2}{q^2},t,\ma\r) \Phi(t)dt,\notag
\end{align}
and
\begin{align}\label{eqdefAO}
A_O(\alpha,\beta,\gamma,\delta,\Phi)=&\frac12\sum_{\ma=0,1} \frac1{\varphi^*(q)}\sum_{d\mid q}\vp(d)\mu\left(\frac qd\r) \sum_{\pm}\sum_{\substack{m-n=\pm h\neq0 \\ (mn,q)=1,~d\mid h}}\frac{\sigma_{\alpha,\beta}(m)\sigma_{\gamma,\delta}(n)} {(mn)^{\frac12}} \\
&\times\int_R\left(1\pm\frac {h}{n}\r)^{-it}V_{\alpha,\beta,\gamma,\delta}\left(\frac{mn} {q^2},t,\ma\r) \Phi(t)dt.\notag
\end{align}
Here the sum $\sum_{\substack{m-n=\pm h\neq0 \\ (mn,q)=1,~d\mid h}}$ is over positive integers $m, n$, and $h$.
Also, we have
\begin{align}\label{eA-1D}
A_{-1}(-\gamma,-\delta,-\alpha,-\beta,\Phi) =&A_{-D}(-\gamma,-\delta,-\alpha,-\beta,\Phi)\\ &+A_{-O}(-\gamma,-\delta,-\alpha,-\beta,\Phi)+O(T^{\ve}q^{\ve}(T/T_0))\notag
\end{align}
with similar expressions for $A_{-D}(-\gamma,-\delta,-\alpha,-\beta,\Phi)$ and $A_{-O}(-\gamma,-\delta,-\alpha,-\beta,\Phi)$.

\subsection{The diagonal terms}\label{secdiagonal1}
For the diagonal terms, we insert the definition of $V$ into \eqref{eq+AD} to see
\begin{align}
A_D(\alpha,\beta,\gamma,\delta,\Phi)=&\frac12\sum_{\ma=0,1}\int_R\Phi(t)\frac1{2\pi i}\int_{(1)} \frac{G(s)}{s}q^{2s}g_{\alpha,\beta,\gamma,\delta}(s,t,\ma) \sum_{(n,q)=1} \frac{\sigma_{\alpha,\beta}(n)\sigma_{\gamma,\delta}(n)}{n^{1+2s}}dsdt.\notag
\end{align}
Then, by the Ramanujan identity, the sum over $n$ is
\begin{align}
  \frac{\zeta_q(1+\alpha+\gamma+2s) \zeta_q(1+\alpha+\delta+2s)\zeta_q(1+\beta+\gamma+2s)\zeta_q(1+\beta+\delta+2s)} {\zeta_q(2+\alpha+\beta+\gamma+\delta+4s)},\notag
\end{align}
which has simple poles at $2s=-\alpha-\gamma$, \emph{etc}, while $G(s)$ vanishes at these poles. Therefore, we can move the integral to $\re(s)=-\frac14+\ve$, passing a pole at $s=0$. By the estimate of $g$ in \eqref{lemVg}, the integral on the new path is
\begin{align}
\ll q^{-\frac12+\ve}\int_R t^{-\frac12+\ve}\Phi(t)dt\ll T^{\frac12+\ve}q^{-\frac12+\ve},\notag
\end{align}
and residue at $s=0$ is
\begin{align}
Z_{q}(\alpha,\beta,\gamma,\delta)\int_R\Phi(t)dt.\notag
\end{align}
We summarize this calculation in the following:
\begin{lemma}
\label{lemD}
We have
\begin{align}
A_D(\alpha,\beta,\gamma,\delta,\Phi)=Z_{q}(\alpha,\beta,\gamma,\delta) \int_R\Phi(t)dt+O\left(T^{\frac12+\ve}q^{-\frac12+\ve}\r),
\end{align}
and similarly the contribution of the diagonal terms to $A_{-1}$ is
\begin{align}
A_{-D}(-\gamma,&-\delta,-\alpha,-\beta,\Phi)\\
&=\frac12Z_{q}(-\gamma,-\delta,-\alpha,-\beta) \sum_{\ma=0,1}\int_R X_{\alpha,\beta,\gamma,\delta}(q,t,\ma)\Phi(t)dt+O\left(T^{\frac12+\ve}q^{-\frac12+\ve}\r).\notag
\end{align}
\end{lemma}

\section{Off-diagonal terms and the proof of Theorem \ref{thmLT}}\label{secdivisor}
\subsection{A divisor problem}
Our treatment of off-diagonal terms requires an estimate on a quadratic divisor problem.
\begin{lemma}
\label{lemmain}
Let $F(x_1,x_2,x_3,x_4,x_5)$ be a smooth function supported on $[1,2]^5$ such that
\begin{align}
\frac{\partial{F}^{(j_1+j_2)}}{\partial x_{i_1}^{j_1}\partial x_{i_2}^{j_2}}\ll_{j_1,j_2}(T/T_0)^{j_1+j_2}T^\ve q^\ve\notag
\end{align}
for any $j_1, j_2\ge 0$ and $i_1, i_2=1,2,3,4,5$.
With $M_1,M_2,N_1,N_2,H\ge1$, we define
\begin{align}
S_d^{\pm}=\sum_{\substack{m_1m_2-n_1n_2=\pm h\neq0 \\ (m_1m_2n_1n_2,q)=1 \\ d\mid h}} F\left(\frac hH,\frac{m_1}{M_1},\frac{m_2}{M_2}, \frac{n_1}{N_1},\frac{n_2}{N_2}\r),\notag
\end{align}
where the sum runs over positive integers $m_1,m_2,n_1,n_2$ and $h$. Suppose that $M_1\le M_2T^\ve q^\ve$, $N_1\le N_2T^\ve q^\ve$ and $H=o\left((M_1M_2N_1N_2)^{\frac12}\r)$. We have
\begin{align}\label{eqSdpm}
S_d^{\pm}=&\sum_{d_1,d_2\mid q}\frac{\mu(d_1)\mu(d_2)}{[d_1,d_2]} \sum_{\substack{m_1,n_1,h \\ (m_1n_1,q)=1}}\frac{k^2h\Delta }{m_1n_1}\int_0^\infty F\left(\frac{kh\Delta }{H},\frac{m_1}{M_1},\frac{kh\Delta (x\pm1)}{m_1M_2}, \frac{n_1}{N_1},\frac{kh\Delta x}{n_1N_2}\r)dx+\mathcal{E},
\end{align}
where $k=(m_1,n_1)$ , $\Delta=[d,(d_1,d_2)] $,
and
\begin{align}\label{lemdE}
\mathcal{E}\ll\frac{H}d N_1^{\frac12}q_0^{\frac12}(M_1+N_1)(T/T_0)^2T^\ve q^\ve.
\end{align}
Here $q_0$ is defined as in Theorem \ref{thmmainT} that $q_0=\max\{d: d\mid q^*, d<{q^*}^{\frac12}\}$ with $q^*=\prod_{p\mid q}p$.
\end{lemma}

\begin{remark}
In \eqref{eqSdpm}, the integral over $x$ in $S_d^{-}$ is actually on $x>1$ since $F$ is supported on $[1,2]^5$. By making the change of variables $x\rightarrow x+1$, we can obtain another form for $S_d^{-}$ that
\begin{align}
S_d^{-}=&\sum_{d_1,d_2\mid q}\frac{\mu(d_1)\mu(d_2)}{[d_1,d_2]} \sum_{\substack{m_1,n_1,h \\ (m_1n_1,q)=1}}\frac{k^2h\Delta }{m_1n_1}\int_0^\infty F\left(\frac{kh\Delta }{H},\frac{m_1}{M_1},\frac{kh\Delta x}{m_1M_2}, \frac{n_1}{N_1},\frac{kh\Delta (x+1)}{n_1N_2}\r)dx+\mathcal{E}.\notag
\end{align}
This expression looks symmetrical with respect to the expression of $S_d^{+}$.
\end{remark}

\begin{proof}
The condition $H=o\left((M_1M_2N_1N_2)^{\frac12}\r)$ implies that $M_1M_2\asymp N_1N_2$. We first consider the sum over the large variables $m_2$, $n_2$, where we rewrite the coprime condition $(m_2n_2,q)=1$ in terms of the M\"obius function that
\begin{align}\label{eqsumF}
\sum_{\substack{(m_2n_2,q)=1 \\ m_1m_2-n_1n_2=\pm h}}F\left(\frac{h}{H},\frac{m_1}{M_1}, \frac{m_2}{M_2}, \frac{n_1}{N_1},\frac{n_2}{N_2}\r)=&\sum_{d_1,d_2\mid q}\mu(d_1)\mu(d_2)\\
&\times\sum_{\substack{m_2,n_2 \\ d_1m_1m_2-d_2n_1n_2=\pm h}} F\left(\frac{h}{H},\frac{m_1}{M_1},\frac{d_1m_2}{M_2}, \frac{n_1}{N_1},\frac{d_2n_2}{N_2}\r).\notag
\end{align}
Let $d_{12}=(d_1,d_2)$. Applying \eqref{eqsumF} with variable changes $d_1\rightarrow d_1d_{12}$, $d_2\rightarrow d_2d_{12}$, we rewrite
\begin{align}
S_d^{\pm}=\mathop{\sum_{d_1d_{12}\mid q}\sum_{d_2d_{12}\mid q}}_{(d_1,d_2)=1}\mu(d_1d_{12})\mu(d_2d_{12})S_d^{\pm}(d_1,d_2,d_{12}),\notag
\end{align}
where
\begin{align}
S_d^{\pm}(d_1,d_2,d_{12})=\sum_{\substack{m_1,m_2,n_1,n_2,h \\ d_1m_1m_2-d_2n_1n_2=\pm h/{d_{12}} \\ (m_1n_1,q)=1, d\mid h}} F\left(\frac{h}{H},\frac{m_1}{M_1},\frac{d_1d_{12}m_2}{M_2}, \frac{n_1}{N_1},\frac{d_2d_{12}n_2}{N_2}\r).\notag
\end{align}
By the definition of $q_0$, there must be $\min\{d_1,d_2\}\le q_0$. Without loss of generality, we would focus ourself on the evaluation of $S_d^{\pm}(d_1,d_2,d_{12})$ with $d_2\le q_0$, and the treatment of the other case is identical.

Since $(d_1,d_2)=1$ and $(m_1n_1,q)=1$, we have $(d_1m_1,d_2n_1)=(m_1,n_1)=k$. Now there is no restriction on the sum over $m_2$ and $n_2$ except for the identity
$d_1m_1m_2-d_2n_1n_2=\pm h/d_{12}$, which we can rewrite as $m_2\equiv(\pm h/kd_{12})\ol{d_1m_1/k} \pmod {d_2n_1/k}$ to eliminate the variable $n_2$. This yields that the sum over $m_2, n_2$ in $S_d^{\pm}(d_1,d_2,d_{12})$ is equal to
\begin{align}
&\sum_{m_2\equiv(\pm h/kd_{12})\ol{d_1m_1/k} ~(\bmod {d_2n_1/k})} F\left(\frac{h}{H},\frac{m_1}{M_1},\frac{d_1d_{12}m_2}{M_2}, \frac{n_1}{N_1},\frac{d_1d_{12}m_1m_2\mp h}{n_1N_2}\r).\notag
\end{align}

Observing that $d_{12}k\mid h$, we make the variable change $h\rightarrow d_{12}kh$, and then the condition $d\mid h$ in $S_d^{\pm}(d_1,d_2,d_{12})$ evolves into $d\mid d_{12}h$ as $(d,k)=1$.
We apply Possion's summation formula to the sum over $m_2$ to see
\begin{align}\label{PossionSk}
S_d^{\pm}(d_1,d_2,d_{12})=\sum_{\substack{m_1,n_1,h \\ (m_1n_1,q)=1 \\ d\mid {d_{12}}h}} \sum_{l\in\mathbb{Z}}e\left(\mp lh\frac{\ol{d_1m_1/k}}{d_2n_1/k}\r) \mathcal{F}_{\pm}(k,d_1,d_{12},h,m_1,n_1,l),
\end{align}
where
\begin{align}
\mathcal{F}_{\pm}&=\frac{k}{d_2n_1}\int_0^\infty
F\left(\frac{d_{12}kh}{H},\frac{m_1}{M_1},\frac{d_1d_{12}x}{M_2}, \frac{n_1}{N_1},\frac{d_1d_{12}m_1x\mp d_{12}kh}{n_1N_2}\r)
e\left(\frac{klx}{d_2n_1}\r)dx\notag\\
&=\int_0^\infty F\left(\frac{d_{12}kh}{H},\frac{m_1}{M_1},\frac{d_1d_2d_{12}n_1x}{kM_2}, \frac{n_1}{N_1},\frac{d_1d_2d_{12}m_1x}{kN_2}\mp\frac{d_{12}kh}{n_1N_2}\r)e\left(lx\r)dx.\notag
\end{align}
Since $F$ is supported on $[1,2]^5$, the integral is actually on the range
\[
x\asymp \frac{kM_2}{d_1d_2d_{12}N_1}\asymp\frac{kN_2}{d_1d_2d_{12}M_1}.
\]

The contribution of the term $l=0$ is
\begin{align}
S^{*\pm}_{d}(d_1,d_2,d_{12})=\sum_{\substack{m_1,n_1,h \\ (m_1n_1,q)=1 \\ d\mid d_{12}h}}\int_0^\infty F\left(\frac{d_{12}kh}{H},\frac{m_1}{M_1},\frac{d_1d_2d_{12}n_1x}{kM_2}, \frac{n_1}{N_1},\frac{d_1d_2d_{12}m_1x}{kN_2}\mp\frac{d_{12}kh}{n_1N_2}\r)dx.\notag
\end{align}
After a variable change $h\rightarrow h[d,d_{12}]/d_{12}=h\Delta/d_{12}$, this evolves into
\begin{align}
\sum_{\substack{m_1,n_1,h \\ (m_1n_1,q)=1}}\int_0^\infty F\left(\frac{kh\Delta }{H},\frac{m_1}{M_1},\frac{d_1d_2d_{12}n_1x}{kM_2}, \frac{n_1}{N_1},\frac{d_1d_2d_{12}m_1x}{kN_2}\mp\frac{kh\Delta }{n_1N_2}\r)dx,\notag
\end{align}
which would contribute to the main term.

For the terms $l\neq0$, integrating by parts $j$ times shows
\begin{align}
\mathcal{F}_{\pm}(k,d_1,d_{12},h,m_1,n_1,l)&\ll T^\ve q^\ve\frac1{l^j} \left(\frac{d_1d_2d_{12}n_1}{kM_2}+\frac{d_1d_2d_{12}m_1}{kN_2}\r)^j (T/T_0)^j\frac{kM_2}{d_1d_2d_{12}n_1}\notag\\
&\ll T^\ve q^\ve\left(\frac{d_1d_2d_{12}N_1}{klM_2}\r)^j(T/T_0)^j\frac{kM_2}{d_1d_2d_{12}n_1}\notag
\end{align}
for any fixed $j\ge0$. This indicates that we can restrict the sum in \eqref{PossionSk} to $0\le|l|\le L$ with
\begin{align}
L=\frac{d_1d_2d_{12}N_1}{kM_2}(T/T_0)T^\ve q^\ve.\notag
\end{align}
Thus, \eqref{PossionSk} evolves into
\begin{align}
S_d^{\pm}(d_1,d_2,d_{12})=S^{*\pm}_{d}(d_1,d_2,d_{12})+\mathcal{E}'\notag
\end{align}
with
\begin{align}
\label{E'}
\mathcal{E}'&=\sum_{k\le H/d}\mathop{\sum_{d_1d_{12}\mid q}\sum_{d_2d_{12}\mid q}}_{(d_1,d_2)=1}\mu(d_1d_{12})\mu(d_2d_{12})\int_0^\infty\sum_{\substack{m_1,n_1,h \\ (m_1,n_1)=1 \\ (m_1n_1,q)=1 \\ d\mid {d_{12}}h}} \sum_{0<|l|\le L}\\
&\times F\left(\frac{d_{12}kh}{H},\frac{m_1k}{M_1},\frac{d_1d_2d_{12}n_1x}{M_2}, \frac{n_1k}{N_1},\frac{d_1d_2d_{12}m_1x}{N_2}\mp\frac{d_{12}h }{n_1N_2}\r)e\left(\mp lh\frac{\ol{d_1m_1}}{d_2n_1}\r) e\left(lx\r)dx,\notag
\end{align}
where we have made variable changes $m_1\rightarrow m_1k$ and $n_1\rightarrow n_1k$.
Note that
$\frac{\partial F}{\partial m_1}\ll \frac k{M_1}(T/T_0)T^\ve q^\ve\ll m_1^{-1}(T/T_0)T^\ve q^\ve $ for $x\asymp\frac{kM_2}{d_1d_2d_{12}N_1}$.
After a summation by parts with the Weil bound for Kloosterman sums, we have
\begin{align}
\sum_{\substack{m_1 \\ (m_1,n_1q)=1}}& F\left(\frac{d_{12}kh}{H},\frac{m_1k}{M_1},\frac{d_1d_2d_{12}n_1x}{M_2}, \frac{n_1k}{N_1},\frac{d_1d_2d_{12}m_1x}{N_2}\mp\frac{d_{12}h}{n_1N_2}\r) e\left(\mp lh\frac{\ol{d_1m_1}}{d_2n_1}\r)\notag\\
&\ll  (lh,n_1d_2)n_1^{\frac12}d_2^{\frac12}\left(1+\frac{M_1}{N_1}\r)(T/T_0) T^\ve q^\ve.\notag
\end{align}
With this in \eqref{E'}, a direct calculation shows that
\begin{align}
\mathcal{E}'&\ll  T^\ve q^\ve \sum_{k\le H/d}\mathop{\sum_{d_1d_{12}\mid q}\sum_{d_2d_{12}\mid q}}_{(d_1,d_2)=1} \sum_{\substack{h\le H/{d_{12}}k \\ d\mid {d_{12}}h}}\sum_{\substack{0<|l|\le L \\ n_1\ll N_1/k}} (lh,n_1d_2)n_1^{\frac12}d_2^{\frac12}\left(1+\frac{M_1}{N_1}\r)(T/T_0) \frac{kM_2}{d_1d_2d_{12}N_1}\notag\\
&\ll   \frac{H}d N_1^{\frac12}q_0^{\frac12}(M_1+N_1)(T/T_0)^2 T^\ve q^\ve\notag
\end{align}
for $d_2\le q_0$. This gives the error term of \eqref{lemdE}.

When $d_1\le q_0$, an identical treatment shows that
\begin{align}
S_d^{\pm}(d_1,d_2,d_{12})=& \sum_{\substack{m_1,n_1,h \\ (m_1n_1,q)=1}}\int_0^\infty F\left(\frac{kh\Delta }{H},\frac{m_1}{M_1},\frac{d_1d_2d_{12}n_1x}{kM_2}, \frac{n_1}{N_1},\frac{d_1d_2d_{12}m_1x}{kN_2}\mp\frac{kh\Delta }{n_1N_2}\r)dx+\mathcal{E}.\notag
\end{align}
The only difference is to eliminate the variable $m_2$ first, and then we should apply Possion's summation formula to the sum over $n_2$ instead.

In conclusion, we sum $S_d^{\pm}(d_1,d_2,d_{12})$ over all possible values of $d_1, d_2$, and $d_{12}$ to get
\begin{align}
S_d^{\pm}=&\mathop{\sum_{d_1d_{12}\mid q}\sum_{d_2d_{12}\mid q}}_{(d_1,d_2)=1}\mu(d_1d_{12})\mu(d_2d_{12})\notag\\
&\times \sum_{\substack{m_1,n_1,h \\ (m_1n_1,q)=1}}\int_0^\infty F\left(\frac{kh\Delta }{H},\frac{m_1}{M_1},\frac{d_1d_2d_{12}n_1x}{kM_2}, \frac{n_1}{N_1},\frac{d_1d_2d_{12}m_1x}{kN_2}\mp\frac{kh\Delta }{n_1N_2}\r)dx+\mathcal{E}.\notag
\end{align}
Making a variable change $x\rightarrow \frac{k^2h\Delta }{d_1d_2d_{12}m_1n_1}(x\pm1)$ in the integral, we have
\begin{align}
S_d^{\pm}=\mathop{\sum_{d_1d_{12}\mid q}\sum_{d_2d_{12}\mid q}}_{(d_1,d_2)=1}&\frac{\mu(d_1d_{12})\mu(d_2d_{12})}{d_1d_2d_{12}} \sum_{\substack{m_1,n_1,h \\ (m_1n_1,q)=1}}\frac{k^2h\Delta }{m_1n_1}\notag\\
&\ \ \ \ \ \ \ \ \ \ \ \ \ \ \ \ \times\int_0^\infty F\left(\frac{kh\Delta }{H},\frac{m_1}{M_1},\frac{kh\Delta (x\pm1)}{m_1M_2}, \frac{n_1}{N_1},\frac{kh\Delta x}{n_1N_2}\r)dx+\mathcal{E}.\notag
\end{align}
This would establish the lemma if we rewrite $d_1d_{12}$ as $d_1$ and $d_2d_{12}$ as $d_2$ in the sum.
\end{proof}
\subsection{Evaluation of $A_{O}$ and $A_{-O}$}
In this section, we produce asymptotic formulae for $A_O$ and $A_{-O}$. Before doing this, we present here a lemma required in following calculation.
\begin{lemma}\label{lemdivisorq}
For any $s$, we have
\begin{align}\label{eqdivisorq}
  \sum_{d\mid q}\vp(d)\mu\left(\frac qd\r)\sum_{d_1,d_2\mid q}\frac{\mu(d_1)\mu(d_2)}{[d_1,d_2]\Delta ^{s}}=\varphi^*(q)q^{-s}\prod_{p\mid q}\left(1-\frac1{p^{1-s}}\r),
\end{align}
where $\Delta=[d,(d_1,d_2)]$.
\end{lemma}
\begin{proof}
Since both sides of \eqref{eqdivisorq} are multiplicative functions on $q$, we just check the identity for prime power. If $q=p$ is a prime, the left-hand side of \eqref{eqdivisorq} is
\begin{align}
  (p-1)\left(\frac1{p^s}-\frac1{p^{1+s}}\right)-\left(1-\frac2p+\frac1{p^{1+s}}\right)=(p-2)p^{-s}\left(1-\frac1{p^{1-s}}\right),\notag
\end{align}
and the identity holds obviously.
If $q=p^m$ with $m\ge2$, the left-hand side of \eqref{eqdivisorq} is
\begin{align}
  \varphi(p^m)\left(\frac1{p^{ms}}-\frac1{p^{1+ms}}\right)-\varphi(p^{m-1})\left(\frac1{p^{(m-1)s}}-\frac1{p^{1+(m-1)s}}\right)=\varphi(p^{m-1})p^{-ms}(p-1) \left(1-\frac1{p^{1-s}}\right),\notag
\end{align}
which is equal to the right-hand side too. Combining these two cases would establish the lemma.
\end{proof}

We specify our asymptotic formulae for $A_O$ and $A_{-O}$ in the following:
\begin{lemma} \label{lemAOA-O}
Let $A_O$ and $A_{-O}$ be defined as before. We have
\begin{align}\label{eAO}
A_O(\alpha,\beta,\gamma,\delta,\Phi)=&\mathcal{M}_{\alpha,\beta,\gamma,\delta}(\Phi) +\mathcal{M}_{\beta,\alpha,\gamma,\delta}(\Phi)\\ +&\mathcal{M}_{\alpha,\beta,\delta,\gamma}(\Phi) +\mathcal{M}_{\beta,\alpha,\delta,\gamma}(\Phi)+O\left( T^{\frac34+\ve}(q/q_0^2)^{-\frac14+\ve}(T/T_0)^3+T^\ve q^\ve\r),\notag
\end{align}
where
\begin{align}
\mathcal{M}_{\alpha,\beta,\gamma,\delta}(\Phi) =&\frac{\zeta_q(1+\alpha-\beta)\zeta_q(1+\gamma-\delta)} {\zeta_q(2+\alpha-\beta+\gamma-\delta)}\int_R \Phi(t)\left(\frac{tq}{2\pi}\r)^{-\beta-\delta}dt\notag\\
&\times\frac1{2\pi i}\int_{(\ve)} \frac{G(s)}{s} \zeta_q(1-\beta-\delta-2s) \zeta_q(1+\alpha+\gamma+2s)ds\notag.
\end{align}
Also,
\begin{align}\label{eA-O}
A_{-O}(-\gamma,-\delta,-\alpha,-\beta,\Phi)= & \widetilde{\mathcal{M}}_{-\gamma,-\delta,-\alpha,-\beta}(\Phi) +\widetilde{\mathcal{M}}_{-\delta,-\gamma,-\alpha,-\beta}(\Phi) +\widetilde{\mathcal{M}}_{-\gamma,-\delta,-\beta,-\alpha}(\Phi)\\ & +\widetilde{\mathcal{M}}_{-\delta,-\gamma,-\beta,-\alpha}(\Phi)+O\left( T^{\frac34+\ve}(q/q_0^2)^{-\frac14+\ve}(T/T_0)^3+T^\ve q^\ve\r),\notag
\end{align}
where, for example,
\begin{align}\label{eqwtM}
\widetilde{\mathcal{M}}_{-\delta-\gamma-\beta-\alpha}(\Phi) =&\frac{\zeta_q(1+\alpha-\beta)\zeta_q(1+\gamma-\delta)} {\zeta_q(2+\alpha-\beta+\gamma-\delta)}\int_R \Phi(t)\left(\frac{tq}{2\pi}\r)^{-\beta-\delta}dt\\
&\times\frac1{2\pi i}\int_{(\ve)} \frac{G(s)}{s} \zeta_q(1-\beta-\delta+2s)\zeta_q(1+\alpha+\gamma-2s)ds.\notag
\end{align}
\end{lemma}

\begin{proof}
We focus ourself on the evaluation of $A_O$, and the treatment of $A_{-O}$ is identical. Recall the expression of $A_{O}$ in \eqref{eqdefAO}, and we rewrite it as
\begin{align}
A_O(\alpha,\beta,\gamma,\delta,\Phi)=&\frac12\sum_{\ma=0,1} \frac1{\varphi^*(q)}\sum_{d\mid q}\vp(d)\mu\left(\frac qd\r) \sum_{\pm}\sum_{\substack{m_1m_2-n_1n_2=\pm h\neq0 \\ (m_1m_2n_1n_2,q)=1,~d\mid h}}\frac{1} {m_1^{\frac12+\alpha}m_2^{\frac12+\beta}n_1^{\frac12+\gamma}n_2^{\frac12+\delta}} \notag\\
&\times\int_R\left(1\pm\frac {h}{n_1n_2}\r)^{-it}V_{\alpha,\beta,\gamma,\delta}\left(\frac{m_1m_2n_1n_2} {q^2},t,\ma\r) \Phi(t)dt.\notag
\end{align}
The estimate \eqref{lemVV} yields that $V(x,t,\ma)$ decays rapidly in $x$ when $x>t^2$, that is to say, the sum over all $m_1m_2n_1n_2\gg (Tq)^{2+\ve}$ gives a negligible contribution $\ll T^{-2020}q^{-2020}$. Also, by integration by parts, there is
\begin{align}
\int_R\left(1\pm\frac {h}{n_1n_2}\r)^{-it}V_{\alpha,\beta,\gamma,\delta}\left(\frac{m_1m_2n_1n_2} {q^2},t,\ma\r) \Phi(t)dt\ll_j\frac{T}{(hT_0/n_1n_2)^j}\notag
\end{align}
for any fixed $j\ge0$, which yields that the contribution of all the terms with $|h|\gg\sqrt{m_1m_2n_1n_2}T_0^{-1}T^\ve q^\ve$ is $O\left(T^{-2020}q^{-2020}\r)$. Hence, we have
\begin{align}
A_O(\alpha,\beta,\gamma,\delta,\Phi)=&\frac12\sum_{\ma=0,1} \frac1{\varphi^*(q)}\sum_{d\mid q}\vp(d)\mu\left(\frac qd\r) \sum_{\pm}\sum_{\substack{m_1m_2n_1n_2\le (Tq)^{2+\ve} \\ m_1m_2-n_1n_2=\pm h \\ 0<h\ll\sqrt{m_1m_2n_1n_2}T_0^{-1}T^\ve q^\ve \\ (m_1m_2n_1n_2,q)=1,~d\mid h}}\frac{1} {m_1^{\frac12+\alpha}m_2^{\frac12+\beta}n_1^{\frac12+\gamma}n_2^{\frac12+\delta}} \notag\\
&\times\int_R\left(1\pm\frac {h}{n_1n_2}\r)^{-it}V_{\alpha,\beta,\gamma,\delta}\left(\frac{m_1m_2n_1n_2} {q^2},t,\ma\r) \Phi(t)dt+O\left(T^{-2020}q^{-2020}\r).\notag
\end{align}

Applying the first partition of unity \eqref{efg}, we rewrite $A_O$ as
\begin{align}\label{eqAO1234}
A_O(\alpha,\beta,\gamma,\delta,\Phi)=A_{O,1}+A_{O,2}+A_{O,3}+A_{O,4} +O\left(T^{-2020}q^{-2020}\r)
\end{align}
with obvious meanings. We will focus on $A_{O,1}$, contributed by $f\left(\frac{m_1}{m_2}\r)f\left(\frac{n_1}{n_2}\r)$, and the treatments for other three terms are identical. Recall that the factor $f\left(\frac{m_1}{m_2}\r)f\left(\frac{n_1}{n_2}\r)$ means that the sum in $A_{O,1}$ is actually over positive integers with $m_1\le m_2$ and $n_1\le n_2$.

We apply the dyadic partition of unity to the sums over $m_1,m_2,n_1,n_2$, and $h$, and it follows that
\begin{align}
A_{O,1}=&\frac12\sum_{\ma=0,1} \frac1{\varphi^*(q)}\sum_{d\mid q}\vp(d)\mu\left(\frac qd\r)\notag \\
&\times\sum_{\substack{M_1M_2N_1N_2\le (Tq)^{2+\ve} \\ M_1\le M_2T^\ve q^\ve ,N_1\le N_2T^\ve q^\ve \\ H\ll\sqrt{m_1m_2n_1n_2}T_0^{-1}T^\ve q^\ve}} \left(S_{d,\ma}^{+}(M_1,M_2,N_1,N_2,H)+S_{d,\ma}^{-}(M_1,M_2,N_1,N_2,H)\r),\notag
\end{align}
where
\begin{align}
S_{d,\ma}^{\pm}(M_1,&M_2,N_1,N_2,H)=\int_R\sum_{\substack {m_1m_2-n_1n_2=\pm h \\ (m_1m_2n_1n_2,q)=1,~d\mid h}}\frac{1} {m_1^{\frac12+\alpha}m_2^{\frac12+\beta}n_1^{\frac12+\gamma}n_2^{\frac12+\delta}}
V_{\alpha,\beta,\gamma,\delta}\left(\frac{m_1m_2n_1n_2} {q^2},t,\ma\r)\notag\\
& \times\left(1\pm\frac {h}{n_1n_2}\r)^{-it}f\left(\frac{m_1}{m_2}\r)f\left(\frac{n_1}{n_2}\r) W\left(\frac{h}{H}\r) W\left(\frac{m_1}{M_1}\r)W\left(\frac{m_2}{M_2}\r)W\left(\frac{n_1}{N_1}\r) W\left(\frac{n_2}{N_2}\r)\Phi(t)dt.\notag
\end{align}
To estimate $S_{d,\ma}^{\pm}(M_1,M_2,N_1,N_2,H)$, we apply Lemma \ref{lemmain} with
\begin{align}
F=&\frac{1}{x_2^{\frac12+\alpha} x_3^{\frac12+\beta} x_4^{\frac12+\gamma} x_5^{\frac12+\delta}} V_{\alpha,\beta,\gamma,\delta}\left(x_2x_3x_4x_5 \frac{M_1M_2N_1N_2} {q^2},t,\ma\r)\left(1\pm\frac{x_1}{x_4x_5}\frac {H}{N_1N_2}\r)^{-it}\notag\\
& \times f\left(\frac{x_2}{x_3}\frac{M_1}{M_2}\r)f\left(\frac{x_4}{x_5}\frac{N_1}{N_2}\r) W\left(x_1\r) W\left(x_2\r)W\left(x_3\r)W\left(x_4\r) W\left(x_5\r).\notag
\end{align}
It is easy to check the condition of Lemma \ref{lemmain} here. Then, it follows that
\begin{align}
S_{d,\ma}^{\pm}(M_1,M_2,N_1,N_2,H)=\mathcal{M}_0^{\pm}(d,\ma)+\mathcal{E}_0,\notag
\end{align}
where
\begin{align}
\mathcal{E}_0\ll\frac{T\mathcal{E}}{M_1^{\frac12+\alpha}M_2^{\frac12+\beta} N_1^{\frac12+\gamma}N_2^{\frac12+\delta}}\notag
\end{align}
with $\mathcal{E}$ given by \eqref{lemdE},
and where
\begin{align}
\mathcal{M}_0^{\pm}(d,\ma)=&\sum_{d_1,d_2\mid q}\frac{\mu(d_1)\mu(d_2)}{[d_1,d_2]} \sum_{\substack{m_1,n_1,h \\ (m_1n_1,q)=1}} \frac{k(kh\Delta )^{-\beta-\delta}} {m_1^{1+\alpha-\beta}n_1^{1+\gamma-\delta}}\int_R\int_0^\infty (x\pm 1)^{-\frac12-\beta}x^{-\frac12-\delta}\notag
\\
&\times V_{\alpha,\beta,\gamma,\delta}\left(\frac{k^2h^2\Delta ^2x(x\pm 1)} {q^2},t,\ma\r)\left(1\pm\frac1x\r)^{-it}f\left(\frac{m_1^2}{kh\Delta (x\pm 1)}\r)f\left(\frac{n_1^2}{kh\Delta x}\r)\notag\\
&\times  W\left(\frac{kh\Delta }{H}\r)W\left(\frac{m_1}{M_1}\r)W\left(\frac{kh\Delta (x\pm 1)}{m_1M_2}\r)W\left(\frac{n_1}{N_1}\r) W\left(\frac{kh\Delta x}{n_1N_2}\r)\Phi(t)dxdt\notag
\end{align}
with $k=(m_1,n_1)$ and $\Delta=[d,(d_1,d_2)]$.

We come to the error term $\mathcal{E}_0$ first, whose contribution to $A_{O,1}$ is bounded by
\begin{align}
&\ll \frac {T^{1+\ve}q^\ve}{\varphi^*(q)}\sum_{\ma=0,1}\sum_{d\mid q}\vp(d)\left(M_1M_2N_1N_2\r)^{-\frac12}\left(\frac{H}d N_1^{\frac12}q_0^{\frac12}(M_1+N_1)(T/T_0)^2\r).\notag
\end{align}
As $H\ll \sqrt{m_1m_2n_1n_2}T_0^{-1}T^\ve q^\ve$ and $M_1,N_1\ll (M_1M_2N_1N_2)^{\frac14}T^\ve q^\ve$, it is bounded by
\begin{align}
\ll \frac {1}{\varphi^*(q)} (M_1M_2N_1N_2)^{\frac38}q_0^{\frac12}(T/T_0)^3T^\ve q^\ve \ll  T^{\frac34+\ve}(q/q_0^2)^{-\frac14+\ve}(T/T_0)^3.\notag
\end{align}

In the summation of $\mathcal{M}_0^{\pm}(d,\ma)$ over $M_1,M_2,N_1,N_2$ and $H$, we may remove the conditions $M_1\le M_2T^\ve q^\ve$, $N_1\le N_2T^\ve q^\ve$, $M_1M_2N_1N_2\le (Tq)^{2+\ve}$, $H\ll \sqrt{m_1m_2n_1n_2}T_0^{-1}T^\ve q^\ve$ with a negligible error, by applying estimates of $f$ and $V$ and integration by parts on $t$ as before. After extending the summation over all $M_1,M_2,N_1,N_2$ and $H$, we remove the dyadic partition of unity to find
\begin{align}
\mathcal{M}_1^{\pm}(d,\ma)=&\sum_{M_1,M_2,N_1,N_2,H}\mathcal{M}_0^{\pm}(d,\ma)\notag\\
=&\sum_{d_1,d_2\mid q}\frac{\mu(d_1)\mu(d_2)}{[d_1,d_2]} \sum_{\substack{m_1,n_1,h \\ (m_1n_1,q)=1}}\frac{k(kh\Delta )^{-\beta-\delta}} {m_1^{1+\alpha-\beta}n_1^{1+\gamma-\delta}}
 \int_R\int_0^\infty (x\pm 1)^{-\frac12-\beta-it}x^{-\frac12-\delta+it}\notag\\
&\times
V_{\alpha,\beta,\gamma,\delta}\left(\frac{k^2h^2\Delta ^2x(x\pm 1)} {q^2},t,\ma\r)f\left(\frac{m_1^2}{kh\Delta (x\pm 1)}\r)f\left(\frac{n_1^2}{kh\Delta x}\r)\Phi(t) dxdt\notag.\notag
\end{align}
Since $V(x,t,\ma)$ is supported on $x>0$, the $x$-integral in $\mathcal{M}_1^{-}(d,\ma)$ is actually over $x>1$. We make the change of variables $x\rightarrow x+1$ in $\mathcal{M}_1^{-}(d,\ma)$, then
\begin{align}
\mathcal{M}_1^{-}(d,\ma)=&\sum_{d_1,d_2\mid q}\frac{\mu(d_1)\mu(d_2)}{[d_1,d_2]} \sum_{\substack{m_1,n_1,h \\ (m_1n_1,q)=1}}\frac{k(kh\Delta )^{-\beta-\delta}} {m_1^{1+\alpha-\beta}n_1^{1+\gamma-\delta}}
 \int_R\int_0^\infty (x+1)^{-\frac12-\delta+it}x^{-\frac12-\beta-it}\notag\\
&\times
V_{\alpha,\beta,\gamma,\delta}\left(\frac{k^2h^2\Delta ^2x(x+1)} {q^2},t,\ma\r)f\left(\frac{m_1^2}{kh\Delta x}\r)f\left(\frac{n_1^2}{kh\Delta (x+1)}\r)\Phi(t) dxdt\notag.\notag
\end{align}
Recalling the definition of $V$ and expressing $f$ in terms of its Mellin transform, we have
\begin{align}\label{eM1M}
\mathcal{M}_1(d,\ma)=&\mathcal{M}_1^{+}(d,\ma)+\mathcal{M}_1^{-}(d,\ma)\\
=&\sum_{d_1,d_2\mid q}\frac{\mu(d_1)\mu(d_2)}{[d_1,d_2]} \frac1{(2\pi i)^3}\int_{(\ve)}\int_{(\ve)}\int_{(1)}\int_R \frac{G(s)}{s}g_{\alpha,\beta,\gamma,\delta}(s,t,\ma)\widehat{f}(u)\widehat{f}(v) q^{2s}\notag\\
&\times \left\{\sum_{\substack{m_1,n_1,h \\ (m_1n_1,q)=1}} \frac{k(kh\Delta)^{-\beta-\delta-2s+u+v}} {m_1^{1+\alpha-\beta+2u}n_1^{1+\gamma-\delta+2v}} \left(J_{+}(s,u,v)+J_{-}(s,u,v)\r)\r\} \Phi(t)dtdsdudv\notag
\end{align}
with
\begin{align}
  J_+(s,u,v)&=\int_0^\infty(x+ 1)^{-\frac12-\beta-s+u-it}x^{-\frac12-\delta-s+v+it}dx,\\
  J_-(s,u,v)&=\int_0^\infty(x+ 1)^{-\frac12-\delta-s+v+it}x^{-\frac12-\beta-s+u-it}dx.
\end{align}
By formula (3.194.3) of \cite{GR65} and the relationship between beta functions gamma functions, we have
\begin{align}
  J_+(s,u,v)=&B(\tfrac12-\delta-s+u+it,~ \beta+\delta+2s-u-v)\notag\\
  =&\frac{\Gamma(\frac12-\delta-s+v+it)\Gamma(\beta+\delta+2s-u-v)}{\Gamma(\frac12+\beta+s-u+it)}\notag
\end{align}
and
\begin{align}
  J_-(s,u,v)=\frac{\Gamma(\frac12-\beta-s+u-it)\Gamma(\beta+\delta+2s-u-v)}{\Gamma(\frac12+\delta+s-v-it)}.\notag
\end{align}
By Stirling's approximation,
\begin{align}
  \frac{\Gamma(\frac12-\delta-s+v+it)}{\Gamma(\frac12+\beta+s-u+it)}=&t^{-\beta-\delta-2s+u+v} \exp\left(\frac{\pi i}{2}(-\beta-\delta-2s+u+v)\r)\notag\\
  &\times\left(1+O\left(\frac{1+|s|^2+|u|^2+|v|^2}{t}\r)\r),\notag
\end{align}
\begin{align}
  \frac{\Gamma(\frac12-\beta-s+v-it)}{\Gamma(\frac12+\delta+s-u-it)}=&t^{-\beta-\delta-2s+u+v} \exp\left(-\frac{\pi i}{2}(-\beta-\delta-2s+u+v)\r)\notag\\
  &\times\left(1+O\left(\frac{1+|s|^2+|u|^2+|v|^2}{t}\r)\r).\notag
\end{align}
Thus, we have
\begin{align}
 J_+(s,u,v)+J_-(s,u,v)=&2\cos\left(\frac{\pi}2 (\beta+\delta+2s-u-v)\r)t^{-\beta-\delta-2s+u+v}\notag \\ &\times\Gamma(\beta+\delta+2s-u-v)\left(1+O\left(\frac{1+|s|^2+|u|^2+|v|^2}{t}\r)\r),\notag
\end{align}
where the contribution of the error $O\left(\frac{1+|s|^2+|u|^2+|v|^2}{t}\r)$ is less than the main term divided by $T$, due to the rapid decay of $G$ and $\widehat{f}$ in $s$, $u$, and $v$.
Solely for notational convenience, we define
\begin{align}
z_1=\beta+\delta+2s-u-v, \ \ \ \ z_2=\alpha-\beta+2u\ \ \ \ \text{and}\ \ \ \ z_3=\gamma-\delta+2v.\notag
\end{align}
Then, the main term of the sum in the brace of \eqref{eM1M} is equal to
\begin{align}
(t\Delta)^{-z_1}\Gamma(z_1)2\cos\left(\frac{\pi z_1}2\r) \sum_{h}\frac1{h^{z_1}}\sum_{\substack{m_1,n_1 \\ (m_1n_1,q)=1}} \frac{k^{1-z_1}} {m_1^{1+z_2}n_1^{1+z_3}}.\notag
\end{align}
Recalling that $\Delta=[d,(d_1,d_2)]$ and $k=(m_1,n_1)$, we may express the last sum over $m_1,n_1$ as an Euler product
\begin{align}
&\prod_{p\nmid q}\Bigg(\sum_{j=0}^\infty\frac{p^{j(1-z_1)}}{p^{j(2+z_2+z_3)}} \sum_{\substack{m,n\ge0 \\ \min\{m,n\}=0}}\frac1{p^{m(1+z_2)+n(1+z_3)}}\Bigg)\notag\\
=&\prod_{p\nmid q}\left(1-\frac1{p^{1+z_1+z_2+z_3}}\r)^{-1} \Bigg(\sum_{m,n\ge0}\frac1{p^{m(1+z_2)+n(1+z_3)}} -\sum_{m,n\ge1}\frac1{p^{m(1+z_2)+n(1+z_3)}}\Bigg)\notag\\
=&\prod_{p\nmid q}\left(1-\frac1{p^{1+z_1+z_2+z_3}}\r)^{-1} \left(1-\frac1{p^{1+z_2}}\r)^{-1} \left(1-\frac1{p^{1+z_3}}\r)^{-1} \left(1-\frac1{p^{2+z_2+z_3}}\r),\notag
\end{align}
which yields
\begin{align}
\sum_{\substack{m_1,n_1 \\ (m_1n_1,q)=1}} \frac{k^{1-z_1}} {m_1^{1+z_2}n_1^{1+z_3}}=\frac{\zeta_q(1+z_1+z_2+z_3)\zeta_q(1+z_2) \zeta_q(1+z_3)}{\zeta_q(2+z_2+z_3)}.\notag
\end{align}
Moreover, the functional equation of the Riemann zeta-function indicates that
\begin{align}
\Gamma(z_1)2\cos\left(\frac{\pi z_1}2\r) \sum_{h}\frac1{h^{z_1}}=(2\pi)^{z_1}\zeta(1-z_1).\notag
\end{align}
Thus, we conclude that
\begin{align}\label{eqM_1(dma)}
&\mathcal{M}_1(d,\ma)=\frac1{(2\pi i)^3}\sum_{d_1,d_2\mid q}\frac{\mu(d_1)\mu(d_2)}{[d_1,d_2]}\int_{(\ve)}\int_{(\ve)}\int_{(1)}\int_R \frac{G(s)}{s}g_{\alpha,\beta,\gamma,\delta}(s,t,\ma) \widehat{f}(u)\widehat{f}(v)\\
&\times\frac{\zeta(1-\beta-\delta-2s+u+v)\zeta_q(1+\alpha+\gamma+2s+u+v) \zeta_q(1+\alpha-\beta+2u)\zeta_q(1+\gamma-\delta+2v)} {\zeta_q(2+\alpha-\beta+\gamma-\delta+2u+2v)}\notag\\
&\times q^{2s}\left(\frac{2\pi}{t[d,(d_1,d_2)]}\r)^{\beta+\delta+2s-u-v} \Phi(t)dtdsdudv\left(1+O\left(\frac1T\r)\right).\notag
\end{align}

Now we come to deduce $A_{O,1}$ from $\mathcal{M}_1(d,\ma)$. We shift the integration in \eqref{eqM_1(dma)} over $u$ and $v$ towards $\re(u)=-1/4+\ve/2$ and $\re(v)=-1/4+\ve/2$. We collect poles from $u=0$ and $v=0$, and for the terms where only one of the two residues is taken we move the other integral to the $(-1/2+\ve)$-line. We do not cross poles at $u=-(\alpha-\beta)/2$ and $v=-(\gamma-\delta)/2$ since we ensured that $\widehat{f}(-(\alpha-\beta)/2)=\widehat{f}(-(\gamma-\delta)/2)=0$. For the integral along the new lines and the residues at only one of $u=0$ and $v=0$, we move the line of integration over $s$ to $\frac14$, and then a direct calculation with the estimate of $g_{\alpha,\beta,\gamma,\delta}(s,t,\ma)$ in \eqref{lemVg} shows that all these are bounded by
\begin{align}\label{eqaerror}
\ll T\sum_{d_1,d_{2}\mid q}\frac{1}{[d_1,d_2]}\left(\frac q{[d,(d_1,d_2)]}\r)^{\frac12}\left(T[d,(d_1,d_2)]\r)^{-\frac12+\ve}.
\end{align}
By summing over $d$, we find that its contribution to $A_{O,1}$ is bounded by
\begin{align}
&\ll T\frac1{\varphi^*(q)}\sum_{d\mid q}\vp(d)\sum_{d_1,d_2\mid q}\frac{1}{[d_1,d_2]}\left(\frac q{[d,(d_1,d_2)]}\r)^{\frac12}(T[d,(d_1,d_2)])^{-\frac12+\ve}\notag\\
&\ll T^{\frac12+\ve}q^{-\frac12+\ve},\notag
\end{align}
which is an acceptable error in the lemma.

For the residue at both $u=0$ and $v=0$, we move the line of the integral over $s$ to $\re(s)=\ve$. After eliminating $g_{\alpha,\beta,\gamma,\delta}(s,t,\ma)$ by \eqref{lemVg}, we observe that it is equal to
\begin{align}
\frac1{2\pi i}\sum_{d_1,d_2\mid q}&\frac{\mu(d_1)\mu(d_2)}{[d_1,d_2]}\int_R\Phi(t)\int_{(\ve)} \frac{G(s)}{s} \left(\frac{tq}{2\pi}\r)^{2s}\left(\frac{2\pi}{t[d,(d_1,d_2)]}\r)^{\beta+\delta+2s}\notag\\
&\times\frac{\zeta(1-\beta-\delta-2s)\zeta_q(1+\alpha+\gamma+2s) \zeta_q(1+\alpha-\beta)\zeta_q(1+\gamma-\delta)} {\zeta_q(2+\alpha-\beta+\gamma-\delta)} dsdt,\notag
\end{align}
adding an error
\begin{align}
\ll \sum_{d_1,d_2\mid q}\frac{1}{[d_1,d_2]}\left(\frac {Tq}{[d,(d_1,d_2)]}\r)^{2\ve}\notag
\end{align}
whose contribution to $A_{O,1}$ is bounded by $\ll  T^\ve q^\ve$ and is acceptable in the lemma.
An arrangement provides that the main contribution of the residue at both $u=0$ and $v=0$ to $A_{O,1}$ is
\begin{align}\label{eqmainAO1}
\frac{\zeta_q(1+\alpha-\beta)\zeta_q(1+\gamma-\delta)} {\zeta_q(2+\alpha-\beta+\gamma-\delta)} \frac1{2\pi i}\int_R \Phi(t)\left(\frac{t}{2\pi}\r)^{-\beta-\delta}\int_{(\ve)} \frac{G(s)}{s} \mathfrak{M}_{\alpha,\beta,\gamma,\delta}(s)dsdt,
\end{align}
where
\begin{align}\label{eqM}
\mathfrak{M}_{\alpha,\beta,\gamma,\delta}(s)=& \zeta(1-\beta-\delta-2s) \zeta_q(1+\alpha+\gamma+2s)\\
&\times\frac{q^{2s}}{\varphi^*(q)}\sum_{d\mid q}\vp(d)\mu\left(\frac qd\r)\sum_{d_1,d_2\mid q}\frac{\mu(d_1)\mu(d_2)}{[d_1,d_2][d,(d_1,d_2)]^{\beta+\delta+2s}}.\notag
\end{align}
We execute the sum in \eqref{eqM} by applying Lemma \ref{lemdivisorq}. It then follows that
\begin{align}
\mathfrak{M}_{\alpha,\beta,\gamma,\delta}(s)= \zeta_q(1-\beta-\delta-2s) \zeta_q(1+\alpha+\gamma+2s) q^{-\beta-\delta}.\notag
\end{align}
Inserting this into \eqref{eqmainAO1} provides the main term of $A_{O,1}$.
In conclusion, we have
\begin{align}
A_{O,1}=\mathcal{M}_{\alpha,\beta,\gamma,\delta}(\Phi) +O\left(T^{\frac34+\ve}(q/q_0^2)^{-\frac14+\ve}(T/T_0)^3 +T^\ve q^\ve\r).\notag
\end{align}
There are similar expressions for $A_{O,2}, A_{O,3}$, and $A_{O,4}$, and  applying these into  \eqref{eqAO1234} gives
\eqref{eAO} immediately. On the other hand, the proof of formula \eqref{eA-O} would be identical after applying Stirling's approximation \eqref{lemVX} to $X_{\alpha,\beta,\gamma,\delta}(q,t,\ma)$ at the beginning.
\end{proof}

\subsection{Assembling the main terms and proving the theorem}
In this section, we prove Theorem \ref{thmLT} by combining all main terms from off-diagonal terms and diagonal terms.
We first deduce the main term for the off-diagonal terms from the asymptotic formulae of $A_O$ and $A_{-O}$ stated in Lemma \ref{lemAOA-O}.
Making the change of variables $s\rightarrow -s$ in $\widetilde{\mathcal{M}}_{-\delta-\gamma-\beta-\alpha}(\Phi)$ and then combining it with $\mathcal{M}_{\alpha,\beta,\gamma,\delta}(\Phi)$, we have
\begin{align}
&\mathcal{M}_{\alpha,\beta,\gamma,\delta}(\Phi) +\widetilde{\mathcal{M}}_{-\delta-\gamma-\beta-\alpha}(\Phi) =Z_q(\alpha,-\delta,\gamma,-\beta,q) \int_R \Phi(t)\left(\frac{tq}{2\pi}\r)^{-\beta-\delta}dt+O\left(T^\ve q^\ve\r)\notag
\end{align}
by the residue theorem, where the poles of the Riemann zeta-function are canceled by $G(\frac{\alpha+\gamma}2)=0$, etc. After combining all the other terms of $A_{O}$ and $A_{-O}$ in the same way, we conclude that
\begin{align}\label{eqAO-O}
A_O&(\alpha,\beta,\gamma,\delta,\Phi) +A_{-O}(-\gamma,-\delta,-\alpha,-\beta,\Phi)\\ =&Z_{q}(\beta,-\gamma,\delta,-\alpha) \int_R \Phi(t)\left(\frac{tq}{2\pi}\r)^{-\alpha-\gamma}dt +Z_{q}(\alpha,-\gamma,\delta,-\beta) \int_R \Phi(t)\left(\frac{tq}{2\pi}\r)^{-\beta-\gamma}dt\notag\\ &+Z_{q}(\beta,-\delta,\gamma,-\alpha) \int_R \Phi(t)\left(\frac{tq}{2\pi}\r)^{-\alpha-\delta}dt +Z_{q}(\alpha,-\delta,\gamma,-\beta) \int_R \Phi(t)\left(\frac{tq}{2\pi}\r)^{-\beta-\delta}dt\notag\\
&+O\left(T^{\frac34+\ve}(q/q_0^2)^{-\frac14+\ve}(T/T_0)^3 +T^\ve q^\ve\r).\notag
\end{align}

We sum up from \eqref{eMD}, \eqref{eA1D} and \eqref{eA-1D} that
\begin{align}
M(\alpha,\beta,\gamma,\delta,\Phi)=&A_D(\alpha,\beta,\gamma,\delta,\Phi) +A_{-D}(-\gamma,-\delta,-\alpha,-\beta,\Phi)\notag\\
&+A_O(\alpha,\beta,\gamma,\delta,\Phi) +A_{-O}(-\gamma,-\delta,-\alpha,-\beta,\Phi)+O((T/T_0)T^\ve q^\ve).\notag
\end{align}
Together with Lemma \ref{lemD} and \eqref{eqAO-O}, this would establish Theorem \ref{thmLT}.

\section{Proof of Theorem \ref{thmLq}}\label{secthmlq}
In this section, we sketch the proof of Theorem \ref{thmLq}. We follow closely the argument in \cite[Theorem 1.3]{Wu20} and keep track of the difference. The calculation of the main terms for Theorem \ref{thmLq} is identical to \cite[Theorem 1.3]{Wu20} since $t$ does not cause any essential difference here. We bound the quantities $E_{M,N}$ and $E_{\ol{M,N}}$ in \cite[Theorem 3.1]{Wu20} by
\begin{align}\label{boundEMN}
E_{M,N},~E_{\ol{M,N}}\ll T_1^{2+\ve} q^{-\frac12+\theta+\ve} M^{-\frac12} N^{\frac12},
\end{align}
where the extra factor $T_1^{2+\ve}$ comes from the ratios of gamma factors in applying spectral large sieve inequalities. The proof of \eqref{boundEMN} is identical to \cite[Section 9]{Wu20}, and the necessary variation on the ratios of gamma factors is an exercise based on Stirling's approximation.

Now, it remains to bound the quantity $B_{M,N}$ in \cite[(3.4)]{Wu20} with $M$ and $N$ far away from each other. After omitting all harmless parameters such as $\alpha,~\beta,~\gamma,~\delta$, and $\ma$, we recall that
\begin{align}
B_{M,N}=\frac1{\vp^*(q)}\sum_{d\mid q}\vp(d)\mu\left(\frac qd\r) \sum_{\substack{(mn,q)=1, m<n\\ m\equiv  n(\bmod d)}}\frac{d(m) d(n)} {m^{\frac12+it}n^{\frac12-it}}
V\left(\frac{mn}{q^2},t\r)W\left(\frac mM\r) W\left(\frac nN\r).\notag
\end{align}
It is easy to see the trivial bound
\begin{align}\label{tboundBMN}
B_{M,N}\ll q^{-1+\ve}(MN)^{\frac12}.
\end{align}
Together with $MN\ll (T_1q)^{2+\ve}$, this means that Theorem \ref{thmLq} is non-trivial only for $T\ll q^{\frac18-\frac34\theta}$.

We write $T_1=q^{\tau}$ with $0\le\tau\le \frac18-\frac34\theta$. Let
\[
\eta=\tfrac1{14}-\tfrac37\theta-\tfrac{11}7\tau,\ \ \ \ \ \ M=q^{\mu},\ \ \ \ \ N=q^{\nu}.
\]
By \eqref{boundEMN} and  \eqref{tboundBMN}, it remains to show
\[
B_{M,N}\ll q^{-\eta+\ve}
\]
for
\begin{align}\label{range}
2-2\eta\le\mu+\nu\le 2+2\tau,\ \ \ \ \ 1-2\theta-2\eta-4\tau\le \nu-\mu.
\end{align}
With
\[
W_t(x)=x^{it}W(x),\ \ \ N\asymp N_1N_2,\ \ \ N_1=q^{\nu_1},\ \ \ N_2=q^{\nu_2}, \ \ \ \nu_1\le\nu_2,
\]
 an evaluation identical to \cite[Section 10]{Wu20} reduces the problem to bounding
\begin{align}\label{R}
R(d,a)\ll q^{-\eta+\ve},
\end{align}
where
\begin{align}\label{eqR}
R(d,a)=\frac{N_2}{a\vp^*(q)\sqrt{MN}}\sum_{(m,q)=1}d(m)W_t\left(\frac mM\r)\sum_{(n_1,q)=1}\sum_{h\neq0} e\left(\frac{hm\ol{a}\ol{n}_1}{d}\r) W_t\left(\frac {n_1}{N_1}\r)\wh{W_t}\left(\frac {h}{H}\r)
\end{align}
is an analogue of $R(d,a)$ in \cite[(10.3)]{Wu20}. The difference is the $W_t$ function in place of the $W$ function, as a result, a longer range of the $h$-sum with
\[
H=adT_1N_2^{-1}\ll T_1qN_2^{-1}.
\]
Since the $h$-sum vanishes for $N_2>T_1q$, we may assume that
\[
\nu_2\le 1+\tau.
\]

Now we divide the region in \eqref{range} into several parts, according to
\begin{enumerate}
  \item $\nu-\mu\ge 1+2\eta+4\tau$;
  \item $1-2\theta-2\eta-4\tau\le\nu-\mu<1+2\eta+4\tau$,
  \begin{itemize}
           \item $\frac12-\theta-2\eta-3\tau< \nu_1< \frac12+2\eta+\tau$;
           \item $\nu_1\ge \frac12+2\eta+\tau$.
         \end{itemize}
\end{enumerate}
Then for each range, we prove that the estimate \eqref{R} holds.
\subsection{The range with $\nu-\mu$ large}
For the range with $\nu-\mu\ge 1+2\eta+4\tau$, a summation by parts with the Weil bound shows
\[
R(d,a)\ll \frac{N_2HT_1}{aq^{1+\ve}}\left(\frac NM\r)^{-\frac12}\left(d^{\frac12+\ve}+N_1d^{-1}\r).
\]
With
$H=adT_1N_2^{-1},\ \ d\le q$, and $N_1\ll N_2\ll q^{1+\tau}$, it follows that
\[
R(d,a)\ll q^{\frac12+2\tau+\ve}\left(\frac NM\r)^{-\frac12}+q^{3\tau}\left(\frac NM\r)^{-\frac12}\ll q^{-\eta+\ve}.
\]

\subsection{The range with $\nu-\mu$ close to 1}
After combining $m$ and $h$ into a longer variable $l=mh$, we have
\[
R(d,a)\ll\frac{N_2q^\ve}{aq\sqrt{MN}}\sum_{l\le L}\left|\sum_{(n_1,q)=1} e\left(\frac{\ol{a}\ol{n}_1l}{d}\r) W_t\left(\frac {n_1}{N_1}\r)\r|
\]
with
\[
L=MHq^\ve\ll \frac{adT_1M}{N_2}q^\ve.
\]

We bound this double sums with the following lemma; see also \cite[Lemma 10.1]{Wu20} and \cite[Theorem 2.4]{KSWX22}.
\begin{lemma}\label{lemDS}
Let $q$ be a positive integer and $(\alpha_k)$ be a sequence of complex numbers satisfying $\alpha_k\ll k^\ve$. For any positive integers $L,K$, we have
\begin{align}
\sum_{l\le L}\left|\sum_{\substack{k\le K\\ (k,q)=1}} \alpha_k e\left(\frac{al\ol{k}}{q}\r) \r|\ll LKq^\ve\cdot\Delta(L,K,q)\notag
\end{align}
uniformly in $a$ with $(a,q)=1$, where we may take the saving $\Delta(L,K,q)$ freely among
\begin{subequations}\begin{align}
&L^{-\frac{1}2}K^{-\frac{1}4}q^{\frac{1}4}+L^{-\frac{1}2}+q^{-\frac{1}2}+K^{-\frac{1}2},\label{eqDS3}\\
&L^{-\frac12}K^{-1}q^{\frac34}+K^{-1}q^{\frac14}+L^{-\frac{1}2}+q^{-\frac12}+K^{-\frac{1}2}\label{eqDS4}.
\end{align}
\end{subequations}
\end{lemma}

For the range with $1-2\theta-2\eta-4\tau\le\nu-\mu<1+2\eta+4\tau$ and $\frac12-\theta-2\eta-3\tau< \nu_1< \frac12+2\eta+\tau$, we apply \eqref{eqDS3} to have
\[
R(d,a)\ll\frac{N_2q^\ve}{aq\sqrt{MN}}\left(L^{\frac12}N_1^{\frac34}d^{\frac14}+L^{\frac12}N_1+L N_1 d^{-\frac12}+L N_1^{\frac12}\r).
\]
As $L\ll \frac{adT_1M}{N_2}q^\ve$ and $N_1\ll q^{\frac12+2\eta+\tau}=q^{\frac9{14}-\frac67\theta-\frac{15}{7}\tau}\le q$, an easy calculation shows
\[
R(d,a)\ll T_1^{\frac12}q^{-\frac14+\ve}N_1^{\frac14}+T_1 q^\ve\left(\frac NM\r)^{-\frac12}N_1^{\frac12}\ll  q^{-\frac14+\frac12\tau+\frac14\nu_1+\ve}+q^{\tau-\frac12(\nu-\mu)+\frac12\nu_1+\ve}.
\]
Since $\nu-\mu>1-2\theta-2\eta-4\tau$ and $\nu_1< \frac12+2\eta+\tau$, we have
\begin{align*}
&-\tfrac14+\tfrac12\tau+\tfrac14\nu_1\le-\tfrac18+\tfrac12\eta+\tfrac34\tau\le-\eta \ \ \ \text{for}\ \ \ \eta\le\tfrac1{12}-\tfrac12\tau,\\
&\tau-\tfrac12(\nu-\mu)+\tfrac12\nu_1\le -\tfrac14+\theta+2\eta+\tfrac72\tau\le -\eta \ \ \ \text{for}\ \ \ \eta\le\tfrac1{12}-\tfrac13\theta-\tfrac76\eta,
\end{align*}
and thus \eqref{R} holds.

For the remaining range, we have
\begin{align}\label{case21}
&1-2\theta-2\eta-4\tau\le\nu-\mu<1+2\eta+4\tau,\\
\label{case22}
&\tfrac12+2\eta+\tau\le \nu_1\le\tfrac12+\tfrac12\tau+\tfrac14(\nu-\mu)\le\tfrac34+\tfrac12\eta+\tfrac32\tau,
\end{align}
then by \eqref{eqDS4},
\[
R(d,a)\ll\frac{N_2q^\ve}{aq\sqrt{MN}}\left(L^{\frac12}d^{\frac34}+Ld^{\frac14}+L^{\frac12}N_1+LN_1d^{-\frac12}+LN_1^{\frac12}\r).
\]
After a simple calculation with $L\ll \frac{adT_1M}{N_2}q^\ve$ and $N_1\ll q^{\frac34+\frac12\eta+\frac32\tau}=q^{\frac{11}{14}-\frac3{14}\theta+\frac57\tau}\le q$, it follows that
\begin{align*}
R(d,a)&\ll T_1^{\frac12}q^{\frac14+\ve}N_1^{-\frac12}+T_1^{\frac12}q^{-\frac12+\ve}N_1^{\frac12}+T_1q^\ve\left(\frac NM\r)^{-\frac12}N_1^{\frac12}\\
&\ll q^{\frac14+\frac12\tau-\frac12\nu_1+\ve}+q^{-\frac12+\frac12\tau+\frac12\nu_1+\ve}+q^{\tau-\frac12(\nu-\mu)+\frac12\nu_1+\ve}.
\end{align*}
Then by \eqref{case21} and \eqref{case22}, we have
\begin{align*}
&\tfrac14+\tfrac12\tau-\tfrac12\nu_1\le\tfrac14+\tfrac12\tau-\tfrac12\times(\tfrac12+2\eta+\tau)=-\tau,\\
&-\tfrac12+\tfrac12\tau+\tfrac12\nu_1\le-\tfrac18+\tfrac14\eta+\tfrac54\tau\le \tau\ \ \ \text{for}\ \ \ \eta\le\tfrac1{10}-\tau,\\
&\tau-\tfrac12(\nu-\mu)+\tfrac12\nu_1\le \tfrac14+\tfrac54\tau-\tfrac38(\nu-\mu)\le-\tfrac18+\tfrac34\theta+\tfrac34\eta+\tfrac{11}4\tau\le -\eta \ \ \text{for}\ \eta\le\tfrac1{14}-\tfrac37\theta-\tfrac{11}7\tau,
\end{align*}
and then \eqref{R} follows.

\section{ Acknowledgments}
This work is supported in part by the National Natural Science Foundation of China (Grant nos. 12271135, 11871187) and the Fundamental Research Funds for the Central Universities of China.

\end{document}